\documentclass[8pt,leqno]{article}
\usepackage[ansinew]{inputenc}
\usepackage{amsmath,amsthm,amsfonts,amssymb,amscd,epsfig,graphicx}
\usepackage{graphicx}
\usepackage{epsfig}
\usepackage{psfrag}
\usepackage{xcolor}
\usepackage[active]{srcltx}

\theoremstyle{plain}
\newtheorem{thm}{\sc {\bf Theorem}}[section]
\newtheorem{lem}[thm]{\sc {\bf Lemma}}
\newtheorem{coro}[thm]{\sc {\bf Corollary}}
\newtheorem{prop}[thm]{\sc {\bf Proposition}}

\theoremstyle{definition}
\newtheorem{rem}[thm]{\sc {\bf Remark}}

\newtheorem{defi}[thm]{\sc {\bf Definition}}

\newcommand{\lgg}{{\left\langle\right.}}
\newcommand{\rg}{{\left.\right\rangle}}
\newcommand{\ve}{{\varepsilon}}
\newcommand{\pain}{\partial_{\infty}}
\newcommand{\ov}{\overline}

\newcommand{\R}{{\mathbb R}}
\newcommand{\hyp}{{\mathbb H}}
\newcommand{\hi}{{\mathbb H}}
\newcommand{\h}{{\mathbb H}}

\font\sc=cmcsc10
\begin{document}

\title { On the structure of hypersurfaces in $\hyp^n\times \R$ with finite strong total curvature }

\author{Maria Fernanda Elbert and Barbara Nelli}
\maketitle

 \let\thefootnote\relax\footnote{ 
 \indent \indent The authors were partially supported by INdAM-GNSAGA, PRIN-2015A35N9B-011.}

\begin{abstract}

We prove that if  $X:M^n\to\h^n\times\R$, $n\geq 3$, is a an orientable, complete  immersion with    finite strong total curvature, then $X$
is proper and  $M$
is diffeomorphic to a compact manifold $\bar M$ minus a finite
number of points $q_1, \dots q_k$.  Adding some extra hypothesis, including  $H_r=0,$ where $H_r$ is a higher order mean curvature, we obtain more information about the geometry of a neighbourhood of each puncture.

The reader will also find in this paper a classification result for the hypersurfaces of $\hyp^n\times \R$ which satisfy $H_r=0$ and are invariant  by hyperbolic translations and a maximum principle in a half space for these hypersurfaces.

\end{abstract}

\section*{Introduction}

 Different notions of  total curvature  of a manifold  M have been used in the literature. Classically a surface has   {\it finite total curvature} if the norm of the gaussian curvature is integrable  on $M$. On the other hand, a hypersurface $M$ of a Riemannian manifold has  {\it finite extrinsic total curvature} if  the norm of the second fundamental form of  $M$  belongs to  $L^n$. Here, by the norm of the second fundamental form we mean the euclidean norm of the vector formed by the principal curvatures of the hypersurface.
 % The notion of extrinsic curvature is  especially suitable in the case of  non Euclidean ambient manifold. 
 Notice that, in the case of minimal surfaces in  $\R^3,$ the two notions coincide.  
 The geometry of minimal surfaces with finite total curvature have been widely studied (see \cite{MePe} for a survey). A classical  result is  due to Huber and Osserman  \cite{Hu, Oss2}:

\

{\em Let  $M$ be a complete oriented, immersed minimal surface in $\R^3$ with finite total curvature. Then $M$  is conformally equivalent to 
a compact Riemann surface $M$  with a finite number of points removed (called the ends of M). Moreover, the Gauss map extends  meromorphically to the punctures.}

\

The extrinsic total curvature was used,  for example, by M. Anderson \cite{An} in order to generalise the previous result  to  minimal submanifolds  of  the Euclidean  space of higher dimension and by  B. White \cite{Wh}, who dealt  with surfaces of the Euclidean space satisfying  properties different from minimality.

\

Our aim is to somehow generalize  Osserman´s result to hypersurfaces of $\hyp^n\times\R,$ $n>2,$ where 
$\hyp^n$ is the hyperbolic space  of dimension n.  The case $n=2$ has already been addressed in \cite{HaRo},  \cite{HNST}, where the authors prove:

\

(\cite[Theorem 3.1 (c)]{HaRo},  \cite{HNST}) {\em Let $M$ be a complete minimal immersion in $\hyp^2\times\R$ with finite  total curvature. Then $M$  is proper, it is conformally equivalent to 
a compact Riemann surface $M$  with a finite number of points removed (called the ends of M). Moreover the third coordinate of the unit normal vector $n_3$ converges to zero  uniformly at
each puncture. Finally  the asymptotic boundary of each one of its ends can be identified with a special kind of closed polygonal curve in $\partial_\infty (\hyp^2\times\R$).}

\

 In   the case $n>2$, we  first consider  a hypersurface with no pointwise assumption on the curvature. Inspired by the ideas of \cite{DE}, we change the hypothesis on finite extrinsic total curvature by that of finite strong total curvature, i.e., we ask that the norm of the second fundamental form of the hypersurface belongs to a special weighted Sobolev space (see  Section \ref{tot-curv-section} for details). Then, we get the following result (see Theorem \ref{strong-normal}).
 
 \

{\em Let $X:M\to\h^n\times\R$, $n\geq 3$, be an orientable complete hypersurface    finite strong total curvature.  Then:

\begin{itemize}\itemsep=-1pt
\item[{\rm (i)}] The immersion $X$ is proper.

 \item[{\rm (ii)}] $M$
is diffeomorphic to a compact manifold $\ov M$ minus a finite
number of points $q_1, \dots q_k$.

\end{itemize}
}

\

This result  partially generalizes \cite[Theorem 1.1]{DE}. Adding some extra hypotheses  on our hypersurface of $\hyp^n\times\R$, including $H_r=0$, we obtain the geometric behaviour of a neighbourhood of a  puncture. Recall that  $H_r$, $r\in\{1,\dots,n\},$  is the mean curvature of order $r$ of an $n$-hypersurface (see Section 1 for a precise definition). We prove the following.

\

{\em Let $X:M\to\h^n\times\R$, $n\geq 3$, be an orientable complete hypersurface    finite strong total curvature and assume that $H_r=0.$ Let  $E$  be a punctured neighbourhood of one of the $q_i$´s and $N=(N_1,\dots,N_{n+1})$ be  
a  unit normal vector field on $E.$ Let $\Pi_1,\dots\Pi_k$  be an  admissible collection of hyperplanes of  $\hyp^n$ and $P_i$, $i=1,\dots,k,$ the corresponding vertical hyperplanes, such that $\partial E\subset \overline{P(\Pi_1,\dots,\Pi_k)}$. Suppose that
$\partial_{\infty}E\cap (\partial_{\infty}\hyp^n\times\R)\subset \partial_{\infty }(P_1\cup\dots\cup P_k)$.
% with $\partial_{\infty}E\cap (\partial_{\infty}\hyp^n\times\R)\cap \partial_{\infty } P_i\neq \emptyset$, for all $i.$ 
  Then:

\begin{itemize}
 \item[{\rm (iii)}] $E$ is asymptotically close  to $P_1\cup\dots\cup P_k.$

 \item[{\rm (iv)}]   For any  sequence  of points $\{p_m\}\subset E$  converging  to a point in $\partial_{\infty}E,$ 
 the sequence $\{N_{n+1}(p_m)\}$ converges uniformly to zero.
\end{itemize}
} 

\

For the definition of  {\em admissible collection} see Definition \ref{admissible}.

%{\color{red} Eu sugiro trocar esta frase em vermelho por uma das duas escritas em azul abaixo  What we had  in mind to generalise is  the  description  of  the behaviour of finite total curvature minimal surfaces in 
%$\hyp^2\times\R$  that is contained in \cite{HaRo, HNST} and goes as follows.}

\

Notice that when working with $n>2$, one looses the technical support of the complex analysis and with $r>1$, one weakens the technical support given by the theory of quasi-linear PDE. Then, it seems somehow reasonable to require a stronger hypothesis on the curvature in our context.

\

%{\color{blue} This theorem can be viewed as a step towards a generalization of the results of  \cite[Theorem 3.1 (c)]{HaRo} and \cite{HNST}, that describe the behaviour of the minimal surfaces with finite total curvature in $\hyp^2\times\R.$ }

 In order to prove (iii) and (iv), we use as barriers a family of  hypersurfaces with $H_r=0$ which are invariant by hyperbolic translations, that we are able to construct  (Theorem \ref{classification-theorem}).  
 As a by product of our construction,  we prove Theorem \ref{iperplano}, which is a maximum principle at infinity  for properly immersed hypersurfaces with $H_r=0$. To the best of our knowledge, this is the first  maximum principle in a half space for hypersurfaces with $H_r=0.$  For this  part of the article, we were  inspired   by the works \cite{BS1}, \cite{N-SE-T} and \cite{ST2}.
 \

% {\color{red}   Add references 
%about finite total curvature surfaces with no minimality assumption \cite{BFM}.  Vamos fazer isto?}

The paper is organized as follows.  After fixing notations in Section \ref{notation}, in Section \ref{invariant-hypersurfaces-section} we describe the family of hypersurfaces that are invariant by hyperbolic translations. In Section \ref{asymptotic-section}, we analyse the influence of the asymptotic boundary  of  hypersurfaces with $H_r=0,$ on their shape at finite points. Hypersurfaces   with finite strong total curvature are studied in Section \ref{tot-curv-section}, 
with no assumption on $H_r.$ Finally in Section \ref{main-result} we prove your main results Theorems \ref{strong-normal},  \ref{strong-normal-minimal}.

\section{Notations}
\label{notation}

  Let $ M^n$ be an orientable Riemannian $n$-manifold and  let  ${\hyp}^n$ be  hyperbolic space (the simply connected Riemannian manifold with constant sectional curvature equal to -1). Let   $X: M^n\to \hyp^n\times\R$  be an isometric immersion. The image  $X(M)$ is a hypersurface of ${\hyp}^n\times\R$ and we shall identify $X(M)$ with $M$ throughout the paper.

   For each $p\in M$, let $A:T_p M\to T_p M$ be the shape operator of $M$ and  
 $\kappa_1,...,\kappa_ n$ be its eigenvalues corresponding to the eigenvectors $e_1,\ldots,e_n$. The {\em higher order mean curvature}  of $M$ of order $r$ is defined as 
 
$$H_r(p)=\frac{1}{\binom nr}\sum_{i_1<...<i_r}\kappa_{i_1}...\kappa_{i_r},$$

i.e.,  the normalized $r^{\rm th}$ symmetric function of $\kappa_1,...,\kappa_ n.$  When $H_{r}=0$, the immersion is called  $(r-1)$-minimal. Thus, the classical minimal immersions would be the $0$-minimal ones.

  We consider the ball model for the hyperbolic space
 $$
\hyp^n=\{x=(x_1,\ldots,x_n)\in\R^n \ \vert \ x_1^2+\ldots+x_n^2<1\}
$$
endowed with the metric
$$
g_\hyp:=\frac{dx_1^2+\ldots+dx_{n}^2}{\left ( \frac{1-|x|^2}{2}\right )^2}=\frac{dx_1^2+\ldots+dx_{n}^2}{F^2},
$$

 where $|x|$ is the euclidean norm of $x$.
As in \cite{ST}, we define the {\em asymptotic boundary} of $\hyp^n\times\R$ as
$$
\partial_\infty(\hyp^n\times\R)=(\partial_\infty\hyp^n\times\R)\cup (\hyp^n\times\{-\infty,\infty\})\cup(\partial_\infty\hyp^n\times\{-\infty,\infty\}).
$$

\

  Let $\Pi$ be a  totally geodesic  hyperplane in $\hyp^n.$ The asymptotic boundary  of $\Pi$ splits $\partial_{\infty}\hyp^n$ into two connected components. Each component can be identified with a spherical cap of the  $(n-1)$-dimensional unit sphere.  We set $\partial_\infty \hyp^n=S_-^{n-1}\cup S_+^{n-1}$, where $S_-^{n-1}$ and $S_+^{n-1}$ are the closure of the two spherical caps determined by $\Pi$ .

Let $\Omega\subset \hyp^n\times\R$ be a nonempty subset. We say that a point $p_\infty\in \partial_\infty(\hyp^n\times\R)$ is {\it an
asymptotic point of} $\Omega$ if there is a sequence $\{p_n\}$ of points of $\Omega$ converging to $p_\infty$. The set of asymptotic points of $\Omega$,  called the {\em asymptotic boundary} of $\Omega$, is denoted by $\partial_\infty\Omega$.

\

 In what follows, we often identify the slice $\hyp^n\times\{0\}$ with $\h^n.$
 By {\em $vertical$ hyperplane} we mean  a complete totally geodesic
hypersurface 
$\Pi\times \R$, where $\Pi$ is any totally geodesic hyperplane
of $\hyp^n.$ We call a
{\em vertical  halfspace} any component of 
$(\hyp^n \times \R)\setminus P$, where $P$ is a vertical hyperplane. 

\

For a fixed totally geodesic hyperplane $\Pi$
of $\hyp^n\times\{0\},$ let $L^+_{\rho}$ and $L^-_{\rho}$ be the  equidistant hypersurfaces to $\Pi$, at distance $\rho$, in the
slice $\hyp^n\times \{0\}$.   
Denote by $Z_{\rho}^+$ the  closure of the non mean convex side of the cylinder
over the hypersurface $L_{\rho}^+$ in $\hyp^n \times \R$. Analogously, we define $Z_{\rho}^-.$  We will call the set   $C_{\rho}=\hyp^n\times\R\setminus Z_{\rho}^+\cup Z_{\rho}^-$    
  {\it $\rho$-cylinder associated to} $\Pi$ (see Figure \ref{rho-cylinder}).

\begin{figure}[!h]
\centerline{
\includegraphics[scale=0.4]{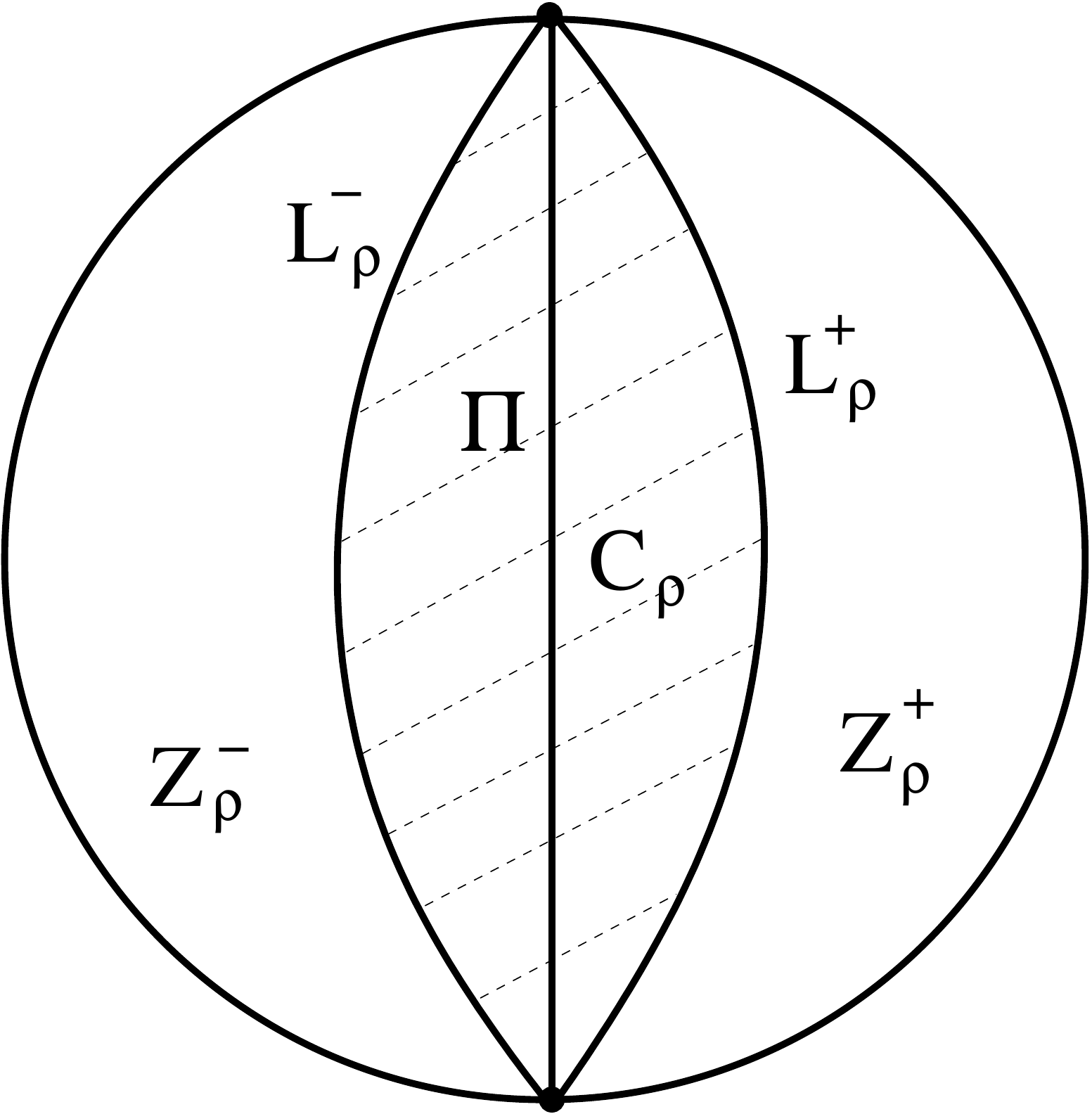}
\label{rho-cylinder} 
}

\caption{$\rho$-cylinder $C_{\rho}$}
\end{figure}

\section{Hypersurfaces  with $H_r=0$ invariant by  hyperbolic  translations }
 \label{invariant-hypersurfaces-section}

 \

We describe a family of hypersurfaces in $\hyp^n\times\R$ with $H_r=0$ which are invariant by a special family of isometries of $\hyp^n\times\R.$

   Let $\gamma$ be a complete geodesic through the origin $\sigma$   of the hyperbolic space $\hyp^n$, parametrized
by the signed distance $\rho$ to  $\sigma$. Let $\Pi$ be the hyperbolic hyperplane orthogonal to $\gamma$ at  
$\sigma$.   For each geodesic  $\beta$ in $\Pi$, passing through $\sigma$,  we consider the hyperbolic translation  along  $\beta$ in  $\hyp^n.$ We notice that the image of any point of $\gamma$ under  the  hyperbolic translations along all geodesics of $\Pi$  passing through $\sigma$ is  an equidistant hypersurface   to  $\Pi$  in $\hyp^n\times\{0\}.$ 
We extend the hyperbolic translation  along  $\beta$ {\em slice-wise} to an isometry of $\hyp^n\times\R$. By  abuse of notation,  this isometry of $\hyp^n\times\R$ will also be called  {\em hyperbolic translation along} $\beta.$

 We show the existence of  a family of hypersurfaces  of  $\hyp^n\times \R$ with $H_r=0$ which are invariant  by hyperbolic translations   along all geodesics of $\Pi$ passing trough $\sigma$.   Moreover we give a complete geometric description of the family. The case  of minimal  hypersurfaces, that is $r=1$, is treated  in \cite {ST2} and 
 \cite{BS1}.

A generating curve parametrized by $ (\tanh(\rho/2),\lambda (\rho))$	in the vertical 2-plane, $\gamma\times\R,$ gives rise, under the previous isometry, to a translationally invariant hypersurface $M$ in $\hyp^n\times\R$  whose intersection with  $\hyp^n \times\{\lambda(\rho)\}$  
 is the equidistant hypersurface to $\Pi\times\{\lambda(\rho)\}$, at distance $\rho.$

The principal directions of the hypersurface $M$ are the tangent vectors to the generating  curve and   to the  equidistant hypersurface. The corresponding principal curvatures are the following (see \cite{BS1}):

 \

\begin{equation}
\label{princ-curv}
\kappa_1=\ddot\lambda(\rho)(1+\dot\lambda(\rho)^2)^{-\frac{3}{2}}, \ \ 
\kappa_2=\dots= \kappa_n=\dot\lambda(\rho)(1+\dot\lambda(\rho)^2)^{-\frac{1}{2}}\tanh(\rho)
\end{equation}

where $(\ \dot{}\ )$ means the derivative with respect to $\rho.$ 
It follows that 

\begin{equation}
\label{Hr}
nH_r\frac{\cosh^{n-1}(\rho)}{\sinh^{r-1}(\rho)}=\frac{\partial }{\partial\rho}\left[\cosh^{n-r}(\rho)\left(\frac{\dot{\lambda}^2}{1+\dot{\lambda}^2}\right )^{\frac{r}{2}}\right].
\end{equation}

We prove the following theorem.

\

\begin{thm}\label{classification-theorem} Let  $\Pi$ be a totally geodesic hyperplane of $\hyp^n\times\{0\}$  passing through $\sigma$ and let $r\leq n$. Then there exists a one parameter family $\{ {\mathcal M}^r_d,\, d>0\}$ of
complete properly embedded  hypersurfaces in $\hyp^n \times \R$, with $H_r=0$, 
invariant under hyperbolic translations  along all the geodesics of $\Pi$ passing through $(\sigma,0).$ For $r=n$, the parameter $d$ assumes also the value $d=0$. The families are described below. 
\begin{enumerate}
\item [{\rm (a)}] $r=n$:\\
\begin{itemize}
\item ${\mathcal M}^r_0$ is a slice.

\item  When $d>0$, ${\mathcal M}^r_d$  is, up to vertical translation,  a complete graph, symmetric with respect to $\Pi$,  whose asymptotic boundary is composed by $((\Pi\cap\partial_\infty\hyp^n)\times\R)\cup(\partial S_-^{n-1}\times\{-\infty\})\cup (\partial S_+^{n-1}\times\{+\infty\})$. Here, $\partial S_{\pm}^{n-1}$ are the hemispheres determined by $\Pi$.
\end{itemize}

\item [{\rm (b)}] $r<n$:\\

\begin{itemize}

\item When $d>1,$   ${\mathcal M}^r_d$ consists of the union of two  vertical
hypersurfaces of finite height, symmetric with respect to $\hyp^n\times \{0\}$, contained in $Z_{\rho_d}^+$, where $\rho_d=\cosh^{-1}(d^{r/n-r})$.  

The asymptotic boundary of ${\mathcal M}^r_d$  is topologically an
$(n-1)$-sphere which is homologically trivial in $\partial_{\infty}
\hyp^n \times \R$. More precisely, if we set $d=\cosh^{\frac{n-r}{r}} a$, we have that
\begin{equation*} 
h_r(d)=\cosh(a)\int_1^{\infty}\frac{dv}{(v^{2q}-1)^{\frac{1}{2}}(\cosh^2(a)v^2-1)^{\frac{1}{2}}} ,
\end{equation*}
 
 is finite and 
the asymptotic boundary of ${\mathcal M}^r_d$ consists
of the union of two copies of an hemisphere $S_+^{n-1}\times
\{0\}$  of $\partial_\infty \hyp^n \times \{0\}$ in parallel slices
$t=\pm h_r(d)$, glued with the finite cylinder 
  $\partial S_+^{n-1}\times [-h_r(d),h_r(d)].$

 The
vertical height of ${\mathcal M}^r_d$ is then $2h_r(d)$. The height of the family
${\mathcal M}^r_d$ is a decreasing function  of $d$ and varies from infinity
(when  $d\longrightarrow 1$)  to  $\frac{\pi r }{(n-r)}$ (when  $d\longrightarrow \infty$). 
\item If $d=1$, then, up to reflection over a slice, ${\mathcal M}^r_1$ consists of a complete  $($non-entire$)$
   vertical graph
over  a halfspace in $\hyp^n\times \{0\}$, bounded by the totally
geodesic hyperplane $\Pi$. It takes infinite  boundary value  on
$\Pi$ and constant value data $c$ on the asymptotic boundary of the halfspace. The asymptotic
boundary of ${\mathcal M}^r_1$ is composed by $\partial_{\infty}\Pi\cap\{\hyp^n\times\{\infty\}\}$, by a  hemisphere  $S_+^{n-1}\times\{c\}$  of
$\partial_\infty \hyp^n\times \{c\}$ and by the a half vertical cylinder
over   $\partial(S_+^{n-1}\times\{c\})$.

    \item If $d<1$, then ${\mathcal M}^r_d$ is an entire vertical graph with finite
vertical height. Its asymptotic boundary consists of a
homologically non-trivial $(n-1)$-sphere in $\partial_\infty \hyp^n
\times \R.$
\end{itemize}

\end{enumerate}

\end{thm}

\begin{proof} 

In each case, we determine the profile curve. The corresponding hypersurface is given by the orbits of the points of the profile curve by the hyperbolic translations along all the geodesics of $\Pi$ passing through the origin of $\hyp^n\times\{0\}.$ The properties of the hypersurfaces will be clear from our description of the profile curve.

Let us first prove (a). By taking $H_r=0$ and $r=n$ in equation \eqref{Hr}, we easily get that $\dot \lambda(\rho)=d,$ for  $d\geq 0$. Then, we have $\lambda (\rho)=d\rho+c$, for a real constant $c.$  

Now, we notice that for $d>0$, this straight line gives a profile curve  in $\gamma\times\R$, parametrized by $(x=\tanh(\rho/2),\lambda(x))$, that is symmetric with respect to $(0,c)$,   is increasing from $x=-1$ to $x=1$ and satisfy

$$
\lim_{x\to -1}\lambda(x)=-\infty\;\;\;\;\mbox{and}\;\;\;\;\lim_{x\to 1}\lambda(x)=+\infty.
$$
 
 This finishes the proof of $(a)$.

\

Now we prove (b). Taking $H_r=0$ in equation \eqref{Hr}, one easily gets that there exists a  constant $d$, with $0<d^r\le \cosh^{n-r}(\rho)$, such that 

\begin{equation} 
\label{eq1}
\cosh^{n-r}(\rho)\left(\frac{\dot{\lambda}^2}{1+\dot{\lambda}^2}\right )^{\frac{r}{2}}=d^r.
\end{equation}

We set $q=\frac{n-r}{r}$ and, by a straightforward computation, we obtain 

\begin{equation}
\label{eq2}
\dot\lambda^2(\rho)=\frac{d^2}{\cosh^{2q}(\rho)-d^2}.
\end{equation}

Here, we can choose $\dot\lambda$ to be the  positive square root in \eqref{eq2} since, up to a reflection across a slice in $\hyp^n\times \R$, the negative root  would give rise  to the same solution.  We divide our study in three cases, depending on $d>1,$ $d=1,$ $0<d<1.$  

\begin{itemize}

\item $d>1$.\\
 Let $a>0$ be such that $d=\cosh^{\frac{n-r}{r}}(a).$ Then, after integration, we have

\begin{equation} 
\label{eq4}
\lambda(\rho)=\int_a^{\rho}\frac{d}{\sqrt{\cosh^{2q}(\xi)-d^{2}}}d\xi.
\end{equation}

By the change of variables $v=\frac{\cosh(\xi)}{\cosh(a)}$, we can rewrite \eqref{eq4} as

\begin{equation} 
\label{eq5}
\lambda(\rho)=\cosh(a)\int_1^{\frac{\cosh(\rho)}{\cosh(a)}}\frac{1}{(v^{2q}-1)^{\frac{1}{2}}(\cosh^2(a)v^2-1)^{\frac{1}{2}}} dv.
\end{equation}

It is easy to see that the integral in \eqref{eq5} converges at $v=1$ and when $\rho$ goes to infinity. 

% in the integral in \eqref{eq5} and we get  that 

%
%\begin{equation} 
%\label{eq5bis}
%\lambda(\rho)=\cosh(a)\int_0^{\left({\frac{\cosh(\rho)}{\cosh(a)}}\right)^{\frac{1}{2q}}}\frac{1}{2qt^{\frac{1}{2}}(\cosh^2(a)(1+t)^{\frac{1}{q}}-1)^{\frac{1}{2}}}(1+t)^{\frac{2q-1}{2q}} dt.
%\end{equation}

 Then we can define

\begin{equation} 
\label{eq6}
h_r(d):=\cosh(a)\int_1^{\infty}\frac{1}{(v^{2q}-1)^{\frac{1}{2}}(\cosh^2(a)v^2-1)^{\frac{1}{2}}} dv
\end{equation}

and  $2h_r(d)$ will be  the height of the hypersurface ${\mathcal M}_d^r.$

A simple computation shows that 
\begin{equation*}
h_r(d)\geq A\int_1^2 [(v-1)(\cosh (a) v-1)]^{-\frac{1}{2}}dt,
\end{equation*}
 
 where $A$ is a positive constant. The latter integral can be computed explicitly and  diverges when 
$a\longrightarrow 0,$ that is when $d\longrightarrow 1.$

%Moreover,  we notice that

 %\begin{equation}
 %\label{eq7}
 %\begin{array}{c}
 %\cosh(a)(\cosh^2 (a)v^2-1)^{-1/2}\geq v^{-1}\\
 %\\
%{\displaystyle\lim_{a\longrightarrow\infty}  \cosh(a)(\cosh^2 (a)v^2-1)^{-1/2}= v^{-1}}\\
%\end{array} \end{equation}

  \

Moreover,  the limit when $a\longrightarrow \infty$  can be taken under the integral and

\begin{equation}
\label{eq9}
\lim_{a\to\infty}h_r(d)=\int_{1}^{\infty} v^{-1}(v^{2q}-1)^{-1/2}dv=\frac{\pi r}{2(n-r)},
\end{equation}

where in the last equality we use that $\int v^{-1}(v^{2q}-1)^{-1/2}dv= \frac{1}{q}\arctan(\sqrt{v^{2q}-1})+{\rm const}.$

 Finally, since

\begin{equation}
\label{S-derivative}
\frac{dh_r}{da}=-\sinh(a)\int_{1}^{\infty} (v^{2q}-1)^{-1/2}(\cosh^2 (a)v^2-1)^{-3/2}dv<0,
\end{equation}

 we conclude that the function $a\to h_r(a)$ decreases from $\infty$ to $\frac{\pi r}{2(n-r)},$  when $a$ increases from $0$ to $\infty$.

\item $d=1.$ 

By replacing  $d=1$ in equation \eqref{eq2} one has that

\begin{equation} 
\label{eq10}
\lambda(\rho)=\int_b^{\rho}\frac{1}{\sqrt{\cosh^{2q}(\xi)-1}} d\xi
\end{equation}

for some constant $b>0.$  It is easy to see that this profile curve tends to $-\infty $, when  $\rho\longrightarrow 0$, and tends to a finite value, when  $\rho\longrightarrow \infty$.

\item $0<d<1.$ 

 In this case one has

\begin{equation} 
\label{eq11}
\lambda(\rho)=d\int_0^{\rho}\frac{1}{\sqrt{\cosh^{2q}(\xi)-d^{2}}}d\xi.
\end{equation}

 The curve  is defined for every value of $\rho> 0$ and   can be extended by symmetry   to values  
$\rho<0.$  Moreover $\lambda$ is bounded. The corresponding hypersurface is an entire vertical graph with finite height.
\end{itemize}

\end{proof}

For future use, we prove a useful property of the hypersurfaces ${\mathcal M}_d^r$. 
 
\begin{prop}\label{sinal}
For a fixed $r<n$ and for any $d>1$, each hypersurface ${\mathcal M}_d^r$ satisfies:

\begin{equation*}
H_j>0, \ \ j<r \ \  {\rm and} \ \  H_j<0, \ \   r<j\leq n.
\end{equation*} 
\end{prop}

\begin{proof}

Let us compute $H_j$ for any $0<j\leq n.$ It is straightforward to see that 

\begin{equation}
\label{Hj}
nH_j=((n-j)\kappa_2+j\kappa_1)\kappa_2^{j-1},
\end{equation}

where $\kappa_1$ and $\kappa_2$ are defined as in \eqref{princ-curv}.

Notice that by deriving \eqref{eq2} one obtains that

\begin{equation}
\label{ddot-lambda}
 \ddot \lambda(\rho)=-\frac{(n-r)}{rd^2}\dot\lambda^3(\rho)\tanh(\rho)(\cosh(\rho))^{\frac{2(n-r)}{r}}.
\end{equation}

By replacing \eqref{princ-curv} and \eqref{ddot-lambda} in \eqref{Hj}, one obtains 

\begin{equation}
\label{hj-final}
H_j= \frac{n(r-j)}{rd^2}\kappa_2^{j-1}\tanh(\rho)(\cosh(\rho))^{\frac{2(n-r)}{r}}
\frac{\dot\lambda^3}{(1+\dot\lambda^2)^{\frac{3}{2}}}.
\end{equation}

This proves the result, since $ \kappa_2$ and $\dot\lambda$ are positive.

\end{proof}

\

Let $\Pi$ and $\bar{\Pi}$ be  totally geodesic  hyperplanes of $\hyp^n\times\{0\}$, where $\Pi$ passes through the origin.  Let $\gamma$ and $\bar{\gamma}$ be the geodesics that are, respectively, orthogonal to $\Pi$ at $\sigma$ and to $\bar{\Pi}$ at a point $p$. Let $\Phi$ be an isometry of the ambient space that takes $\Pi$ into $\bar{\Pi}$, takes $\gamma$ into $\bar{\gamma}$ and   that preserves the $t$-coordinate. We notice that, by applying  $\Phi$ to each family ${\mathcal M}^r_d$ constructed in Theorem \ref{classification-theorem}, we obtain a one parameter family of hypersurfaces invariant under hyperbolic translations along the geodesics of $\bar{\Pi}$ passing through $p$. In the next sections, by abuse of notation,  we will denote by ${\mathcal M}^r_d$ any  hypersurfaces obtained from ${\mathcal M}^r_d$ applying  an isometry like $\Phi.$

\section{Maximum Principle and Asymptotic  Theorems}
\label{asymptotic-section}

In this section, we use the translationally invariant hypersurfaces ${\mathcal M}^r_d$, constructed above, and a maximum principle, in order to investigate
how the boundary behaviour of a hypersurface with $H_r=0$ contained in a halfspace, constrains   the behaviour of the hypersurface at finite points.
Moreover we prove an obstruction result for hypersurfaces with $H_r=0$ and a given boundary.

\smallskip

The suitable version of maximum principle for our purposes is stated below.  For further details about such generalized maximum principles, see \cite{ENS},  \cite{FS} \cite{HL1},\cite{HL2}.

\smallskip
\smallskip
{\bf Maximum Principle} \cite[Theorem  2.a]{FS} {\em Let $M$ and $M'$ two  oriented hypersurfaces with $H_r=H_r'\equiv 0,$ tangent at a point $p,$ with normal vector pointing in the same direction. Suppose that $M$ remains on one side of $M'$  in a neighborhood of $p.$ Suppose further that $H_j'(p)\geq 0,$ $1\leq j\leq r$ and either $H_{r+1}\not=0$ or  $H'_{r+1}\not=0.$  Then $M$ and $M'$ coincide in a neighborhood of $p.$ }

\smallskip
\smallskip

Given $r\in\{1,\dots,n\}$, and a totally geodesic hyperplane $\Pi$ in $\hyp^n\times\{0\},$ we consider the hypersurfaces ${\mathcal M}^r_d$, with $d>1$, described in Section \ref{invariant-hypersurfaces-section} (see the last sentence of  Section \ref{invariant-hypersurfaces-section}). We notice that, by Proposition \ref{sinal}, all the hypersurfaces ${\mathcal M}^r_d$, $d>1$, satisfy the assumptions of $M'$ in the maximum principle.

 Denote by $Q_{\Pi}$   the  halfspace  determined by $\Pi\times\R$ which contains  $Z_{\rho_d}^+.$ Clearly,  any  vertical
translation of the hypersurface ${\mathcal M}^r_d$ is contained in $Z_{\rho_d}^+$ and, moreover, any vertical translation of
${\mathcal M}^r_d$ is arbitrarily close to $\partial Q_\Pi=\Pi\times\R,$  provided $d$ is sufficiently close to one.

\smallskip

The proof of Theorem \ref{iperplano} below is inspired by that of  \cite[Theorems 3.2, 4.5]{N-SE-T}. 

\

\begin{thm}
\label{iperplano}
Let $M$ be a hypersurface, with $H_r=0$, properly immersed in
$\hyp^n\times\R$. Let $P$ be a vertical  hyperplane
and
$P_+$ one of the two halfspaces determined by $P.$ If $\partial M \subset\overline{ P_+}$ ($\partial M$  possibly empty) and 
$\partial_{\infty}M\cap(\partial_{\infty}\hyp^n\times\R)\subset\partial_{\infty} P_+,$  then
$M\subset \overline{ P_+}$.
\end{thm}

\begin{proof}

 Let $\Pi \subset \hi^ n\times\{0\}$ be a totally geodesic hyperplane  and $Q_{\Pi} $ the half space determined by $\Pi$, chosen such that  

\begin{equation*}
(\mathcal{P}) \   \  \ \  \  \     \  \  \  \ Q_{\Pi} \subset  (\h^ n\times \R)\setminus P_+, \  \  \pain\Pi\cap \pain P=\emptyset
\end{equation*}

%  Observe that 
% $\partial M \cup
% \big(\partial_{\infty}M \cap (\pain \hi2 \times \r)\big)
% \subset P_+\cup\partial_{\infty} P_+.$

%Denote by $L$ and $R$ the vertical lines in the asymptotic boundary of $P.$ One
% can assume that the asymptotic boundary of each $M_d,$ $d>1$ is contained   in
%  the connected component of  $\partial_{\infty}\h^2\times \r$ between $L$ and
% $R$ not contained in the asymptotic boundary of $P_+.$

We fix a $d>1$ and we consider the family of hypersurfaces ${\mathcal M}^r_d$  contained in $Z_{\rho_d}^+\subset Q_{\Pi}$. The following two properties hold:

\begin{enumerate}
\item[{\rm (I)}] \label{item.asymptotic}The intersection of $\partial_{\infty}M$ with 
 $\partial_{\infty}(\h^ n\times\R)\setminus \partial_{\infty} P_+$ 
  contains no points at finite
height. 
\item[{\rm (II)}]  The asymptotic boundary of any vertical translation of ${\mathcal M}^r_d$ is  contained in the
asymptotic
boundary of  $ Q_{\Pi}\subset\h^n\times\R\setminus  P_+.$

%\item Any vertical translation of
%$M_d$ is arbitrarily close to $Q_\gamma$ if $d$ is sufficiently close to 1.
\end{enumerate}

We will get the result  by applying the maximum 
principle between the hypersurface $M$ and some  isometric copy of ${\mathcal M}^r_d$'s. 

\smallskip
 Let $\gamma$ be the geodesic in $\hyp^n\times \{0\}$, orthogonal to $\Pi$ at a point $p\in\Pi.$  We parametrize $\gamma$ by the signed distance to $p,$  say $s,$ with orientation pointing towards $Q_{\Pi}.$  For any $s,$ we consider the  isometry of $\h^ n\times\R$ that preserves the $t$-coordinate and  takes $\Pi$ into the geodesic hyperplane orthogonal to $\gamma$ at a distance $s$ from $p$. 
   By letting $s\longrightarrow \infty$ and by applying the above isometries, we obtain a family of hypersurfaces ${\mathcal M}^r_d(s)$, isometric to ${\mathcal M}^r_d$, that collapses to a vertical segment in $(\partial_{\infty}\h^n\times\R)\cap\partial_\infty Q_{\Pi}$ of height $2h_r(d)$.  We claim that, for some $s,$ $M$  and ${\mathcal M}^r_d(s)$  are disjoint. In fact, suppose that, when $s\longrightarrow \infty$, ${\mathcal M}^r_d(s)$ always have a nonempty
intersection with $ M$. Then, there would exists a point of 
the asymptotic boundary of $M$ at finite height in
 $\partial_{\infty}(\h^ n\times\R)\setminus \partial_{\infty}P_+$, giving a contradiction with (I). Then, the claim is proved.  Now, starting with a ${\mathcal M}^r_d(s)$ disjoint from $M,$ 
 we apply the (inverse) isometries to come  back towards  the original position, that is, we let $s\longrightarrow 0$.  Then, either we find a first intersection point between $M$ and ${\mathcal M}^r_d(s)$, for some $s,$  contradicting the maximum principle, or we reach the original position without  touching $M$. The same process can be done with any vertical
translation of ${\mathcal M}^r_d$ and we can conclude that
 $M$ is contained in $\h^ n\times\R\setminus Z_{\rho_d}^+$. 
 
 Now, we let $d\to 1$ and the maximum principle yields that  $M$  is contained in the closed halfspace
$\h^ n\times\R\setminus Q_{\Pi}$. Since this holds for any totally geodesic hyperplane $\Pi$ satisfying 
the property  $(\mathcal P),$  we conclude that $M$ is contained in the closure of 
$P_+$.

\smallskip

\end{proof}

Let us extend Theorem \ref{iperplano} to the case of a more general asymptotic boundary.

%\begin{coro}
%Let $M$ be a complete minimal  hypersurface properly  immersed in
%$\h^n\times\r.$ Let $P$ be a vertical totally geodesic hyperplane.
%If $\partial_{\infty}M\cap(
%\partial_{\infty}\h^n\times\r)\subset\partial_{\infty}P,$ then
%$M=P.$
%\end{coro}

\smallskip

\begin{defi}\label{admissible}
 Let $\Pi_1,\dots\Pi_k$ be a collection of hyperplanes in $\hyp^n$ such that $\partial_{\infty}\Pi_i=S_i$, where for $i=1,\dots,k,$  $S_i$ is an $(n-2)$-sphere in  $\partial_{\infty}\hyp^n$.  We say that the hyperplanes $\Pi_1,\dots\Pi_k$ are an  {\it admissible collection} if it is possible to choose open (n-1)-spheres $B_1,\dots,B_k$  in  $\partial_{\infty}\hyp^n$, bounded by $S_1,\dots S_k$, which are mutually disjoint.
 \end{defi}

\begin{defi}\label{halfspace-def}
Let $\Pi_1,\dots\Pi_k$ be an admissible collection of hyperplanes in $\hyp^n$ and let $P_j=\Pi_j\times\R$, $j=1,\dots,k,$ be the corresponding  vertical hyperplanes in $\hyp^n\times\R$.  Denote by $\tilde P_j$ the half-space such that 
${\displaystyle \cup_{i\neq j}}\left (\Pi_i\times\R\right )\subset\tilde P_j.$  We define $P(\Pi_1,\dots,\Pi_k):=\cap_{i=1}^k\tilde P_i.$
\end{defi}

Notice that $\Pi_i$ and $\Pi_j,$  $i\not=j,$ can meet at most at one point. This yields that $\partial_{\infty}P_i$ and $\partial_{\infty}P_j,$  $i\not=j,$ can meet at most 
at a vertical line.
\

\begin{coro}
\label{halfspace-coro}
Let $M$ be  a complete hypersurface with $H_r=0,$ possibly with finite boundary, properly immersed
in
$\hyp^n\times\R$ and let
$\Gamma=\partial_{\infty}M\cap (\partial_{\infty}\hyp^n\times\R).$ Let $\Pi_1,\dots,
\Pi_k$  an  admissible collection of hyperplanes. If  $\Gamma\subset
\partial_{\infty}P(\Pi_1,\dots, \Pi_k)$ and $\partial M\subset \overline{ P(\Pi_1,\dots, \Pi_k)},$ then $M$ is
contained in
$\overline{P(\Pi_1,\dots, \Pi_k)}$.
\end{coro}

\
  Next result establishes some  obstruction to the existence of a hypersurface in $\hyp^n\times\R$  with $H_r=0:$ 
in particular the shape of the asymptotic boundary of a hypersurface may prevents the hypersurface to have $H_r=0.$
The result  is a  generalization of  \cite[Corollary 2.2]{ST2} and \cite[Theorem 4.6]{ N-SE-T}.

\

\begin{thm}\label{T.slab.catenoid}
Let $S_\infty \subset \partial_{\infty} \hyp^n \times \R$ be a closed set 
whose  vertical projection  on $\partial_{\infty} \hyp^n \times \{0\}$ omits
an open subset. Assume that $S_\infty $ is contained in an open slab whose 
height is equal to $\frac{r\pi}{n-r}$.  Then, there is no  connected  hypersurface $M$ with $H_r=0,$ $\partial M=\emptyset,$
properly
embedded 
 in $\hyp^n \times \R,$  with asymptotic boundary $S_\infty$.

\end{thm}

\begin{proof}

\smallskip

Assume, by contradiction,
that there exists   a
hypersurface  $M$, satisfying the assumptions and with $\partial_{\infty} M=S_{\infty}$.  Then, up to a vertical
translation, we
can assume that $M$ is
contained in the  slab  ${\mathcal{B}}:=\{(p,y)\in \hyp^n \times\R;\; t_0\leq t\leq \frac{r\pi}{n-r}-t_0\}$ for some $t_0>0$ (see \cite[Proposition 3.1]{ENS})
and  $S_\infty\subset \partial_{\infty} \mathcal{B}$.
  As the vertical projection  of $S_{\infty}$  omits
an open subset, say  $U$,  by Theorem \ref{iperplano}, we find   a totally geodesic hyperplane 
$\Pi\subset \hyp^n\times \{0\} $ such that a component, say $\Pi^+$,  of 
$\hyp^n \times \{0\} \setminus \Pi$ satisfies:
\begin{enumerate}
\item $\partial_{\infty}\Pi^+ \subset U$.
\item $M \cap  ( \overline{\Pi}^+ \times \R) =\emptyset$.
\end{enumerate}
% Let $\gamma \subset \hi n\times \{0\}$ be a complete geodesic with an
% asymptotic point $p_\infty$ in $\pain \pi^+$. 
Let ${\mathcal C}\subset \hyp^n \times (0,\frac{r\pi}{n-r})$ be
any $n$-catenoid with $H_r=0$, such that a component of its asymptotic boundary stays strictly above 
$\partial_{\infty}\mathcal{S}$ and the other component stays strictly below $\partial_{\infty}\mathcal{S}$. The existence of such catenoids is proved in \cite[Theorem 2.1]{ENS}. There, it is also proved that the the $j$-mean curvatures of the catenoids satisfy $H_j(p)< 0,$ $1\leq j\leq r$  and $H_{r+1}<0$ (see \cite[Proposition 2.2]{ENS}).

 We define $K={\mathcal B}\cap{\mathcal C}.$ $K$ is compact, connected and its 
boundary lies in the boundary of the slab ${\mathcal{B}}$.

Let $q\in M$ be a point,  let $q_0\in \hyp^n \times \{0\}$ be the vertical
projection of $q$ and let   $p_\infty$ be a point in $ \partial_{\infty}\Pi^+.$ 
 Denote by $\widetilde \gamma\subset  \hyp^n \times \{0\}  $
the complete geodesic passing through $q_0$ such that 
$p_\infty \in \partial_{\infty}\widetilde \gamma$.  We can translate $K$ along $\widetilde \gamma$ (with the usual isometry  of $\hyp^n\times\R$ that preserves the $t$-coordinate),  such that the translated $K$  is contained in
the halfspace $\Pi^+ \times \R$.

Now we  come back translating $K$ towards $M$ along $\widetilde \gamma.$ 
Observe that the boundary of the translated
copies of $K$  does not touch $M$. Therefore, doing the 
translations of $K$ along $\widetilde  \gamma$ we find a first interior point of contact
between $M$ and a translated copy  of $K$. Hence, $M={\mathcal C}$ by the maximum
principle, which leads to a contradiction and  completes the proof.
\end{proof}

\begin{rem}
Similar  techniques may be applied to generalise non existence results analogous to \cite[Theorem 4.6]{N-SE-T},  \cite[Theorem 2.1]{ST2}.

\end{rem}

\section{Finite Strong Total Curvature}
\label{tot-curv-section}

In this section, we deal with  a general isometric immersion  $X: M\longrightarrow\hyp^n\times\R$  without any  assumption on $H_r$.  The notion of strong total curvature, introduced by the first author and M. Do Carmo  \cite{DE} for hypersurfaces in $\R^{n+1}$, is defined as a  special norm  of  the shape operator, $A$, of $M=X(M)$.

\

Let  $p_0$ be a fixed point of  $M$ and denote by $\xi(p)$ the intrinsic distance in $M$ from $p$ to $p_0.$ 

Let $\Omega\subset M$.  Given any $q\geq 1$, we define the following two function spaces.
\begin{itemize} 
\item $L^q _s (\Omega)$ is the  {\it weighted space } of weight $s\in\R$ of all measurable functions of finite norm
$$
||u||_{L^q_s (\Omega)}=\left ( \int_\Omega |u|^q \xi^{-qs-n} \;dM\right)^{1/q}\hspace{-15pt}.
$$
\item  $W^{1,q} _s(\Omega)$ is the {\it weighted Sobolev space} of weight $s$ of all measurable functions of finite norm

\begin{equation*}
||u||_{W^{1,q} _s(\Omega)}=||u||_{L^q _{s} (\Omega)}+||\nabla u||_{L^q _{s-1} (\Omega)},
\end{equation*}
\noindent where $\nabla u$ is the gradient of $u$ in $M$.
\end{itemize}

\

  The latter was used by Bartnik, in a pioneer paper \cite{B}, to   define a suitable decay at infinity of the metric of a manifold (asymptotically flat spaces) that guarantees that the ADM-mass is  a geometric invariant. We  point out that, since then, it was used by a lot of authors and the literature about the subject is wide.

\begin{defi} 
\label{defi-finite-strong}
Let $M$ be a hypersurface of $\hyp^n\times\R$ and $A$ its shape operator. 
We define the quantity $||\,|A|\,||_{W^{1,q} _{-1}(M)}$ to be  the  {\em strong total curvature} of the immersion  $M$   and we say that the immersion has 
{\em finite strong total curvature} if

\begin{equation}
\label{FSTC-def}|A|\in W^{1,q} _{-1}(M),\;\;{\rm for\; some }\;  q>n,
\end{equation}
\end{defi}
where $|A|$ is the norm of the shape operator.

\

Notice that  the definition of strong total curvature does not depend on the choice of the point $p_0$ and  that \eqref{FSTC-def} can be written as follows:
\begin{equation*}
||\,|A|\,||_{W^{1,q} _{-1}(M)}=\left ( \int_M |A|^q \xi^{q-n} \;dM\right)^{1/q}+\left ( \int_M |\nabla |A||^q \xi^{2q-n} \;dM\right)^{1/q}<\infty,\;\;\mbox{for\; some}\;  q>n.
\end{equation*}

We point out that the norm $||\,|A|\,||_{W^{1,q} _{-1}(M)}$ is invariant by dilations of the intrinsic metric of $M.$ 

As in \cite{DE}, we will estimate the rate of the decay at infinity of  $|A|$ (see Proposition \ref{decay}). 

%
%\begin{exem}
%\label{exemple-finite-strong} Let us study the strong total curvature of some known examples.
%
%The hypersurfaces ${\mathcal M}^r_d,$ constructed in Section \ref{invariant-hypersurfaces-section}, have infinite extrinsic total curvature (hence infinite strong total curvature), because they are invariant by hyperbolic translations and are not flat. 
% As  a consequence of  Theorem \ref{strong-normal-minimal}, rotationally invariant hypersurfaces with $H_r=0,$ 
%described in \cite{ENS},  have infinite strong total curvature  but, at  least for $r=1$, have finite extrinsic total curvature (see \cite{BS1}). 
%On the other hand, vertical hyperplanes have finite (strong)  total curvature (indeed zero).
%
%In the case of minimal surfaces in $\hyp^2\times\R,$ one has many examples. 
%The  vertical catenoids, constructed by F. Morabito, M. Rodriguez \cite{MR} and independently by 
%J. Pyo  \cite{P} are non trivial examples with finite total curvature, while, in \cite{ST1} many minimal examples  with infinite total curvature in $\hyp^2\times\R$ are described.
%{\color{red} We should prove that horizontal catenoids in $\hyp^2\times\R$ have finite strong total curvature.} {\color{teal} 
%Para dimensao 2, nao tem muita importancia a strong. nao sei fazer a conta e nao acho que vale a pena. }

%\end{exem}

Next lemma is analogous to  \cite[Lemma 3.1]{DE}.

%
%We say that   a sequence $\{M_i:=X_i(M_i)\}$ of hypersurfaces in $\hyp^n\times\R$ {\em converge   $C^1$ uniformly on compact sets} to a  union of hypersurfaces $M_\infty$  if, for any $p\in M_\infty$ and each tangent plane $T_p M_\infty,$ there exists a ball $B(p)\subset \hyp^n\times\R$ around $p$ so that, for $i$ large, the image
%by $X_i$ of some connected component of $X_i^{-1}(B(p)\bigcap M_i)$  is the graph of  a function $g_i$ over (a part of) $T_p M_\infty$   and the
%sequence $g_i$ converges $C^1$ to the graph of a function $g_\infty$, that define $M_\infty$ over  a part of  $T_p M_\infty$.
%
%{\color{red} I do not understand if you agree with my definition. We can put it before or after. It does not matter to me}
%
%{\color{red} I propose the following:}
%
%{\color{blue} We say that a sequence $M_i$ of hypersurfaces in $\hyp^n\times\R$ {\em converges   $C^1$ uniformly on compact sets} to a  hypersurface $M_\infty$ (possibly not connected) if for any $p\in M_\infty,$ there exists a ball $B(p)$ in $\hyp^n\times\R$ around $p,$ such that for $i$ large,  $M_i\cap B(p)$ is the graph of a function $g_i$ over (a part of) $T_p M_\infty$ in some coordinate system  and the
%sequence $g_i$ converges $C^1$ to the graph of a function $g_\infty$, that define $M_\infty$ over  a part of  $T_p M_\infty$ (one sees $T_pM_\infty$   as a subset of $\hyp^n\times\R$ via the exponential map).}
%

\begin{lem}
\label{compactness-lemma}
Let $B\subset \hyp^n\times\R$ be a bounded domain with smooth boundary $\partial B.$  
Let $\{W_i\}$ be a sequence of connected n-manifolds and let $X_i:W_i\to \hyp^n\times\R$ be hypersurfaces such that $ \partial X_i(W_i)\cap B=\emptyset$ and $ X_i(W_i)\cap B = M_i$ is connected and
nonempty. Assume that there exists a constant $C>0$ such that,  for every $i,$
${\displaystyle\sup_{x\in M_i}|A_i(x)|^2<C},$ where $A_i$ is the shape operator of $M_i,$  and that there exists a sequence of
points $\{x_i\}$, $x_i\in M_i$, with a limit point $x_0\in B$. Then:
\begin{itemize}

\item[\rm i)] A subsequence of $(M_i)$ converges $C^{1,\lambda}$ 
on the compact parts (see the definition below) to   a union of
hypersurfaces $M_\infty \subset B$, where $\lambda<1$.

\item[\rm ii)] If, in addition, $\left ( \int_{M_i} |A|^q \alpha_i \;dM_i\right)^{1/q}+\left ( \int_{M_i} |\nabla |A||^q \beta_i \;dM_i\right)^{1/q}\to 0,$  for sequences $\{\alpha_i\}$ and $\{\beta_i\}$ of continuous functions on $M_i$ satisfying ${\displaystyle\inf_{x\in
M_i}\{\alpha_i,\beta_i\}\geq \kappa}>0$, then a subsequence of $|A_i|$ converges to zero everywhere.
\end{itemize}
\end{lem}

By $C^{1,\lambda}$ convergence to $M_\infty$ on compact sets we mean that for any $m\in
M_\infty$ and each tangent plane $T_m M_\infty$ there exists  a ball
$B(m)$  of $\hyp^n\times\R$  around $m$ so that, for $i$ large, the image
by $X_i$ of some connected component of $X_i^{-1}(B(m)\bigcap M_i)$  is the graph  over a part of  $T_m M_\infty$ of  a function $g_i^m$ and the
sequence $g_i^m$ converges $C^{1,\lambda}$ to the function $g_\infty$, that defines 
$M_\infty$ as a graph over a neighbourhood of $m$ in the chosen plane $T_m M_\infty$.

\begin{proof} In i) and ii) we can work in  compact subsets $B$ of $\hyp^n\times\R$. By using  \cite[Proposition 3.1]{ST}, we can treat  $M_i\cap B$ as a sequence of submanifolds  of  $\R^{n+1}$ with  uniformly bounded  second fundamental form. Then we can use the proof of \cite[Lemma 3.1]{DE} in order to conclude our proof.

\end{proof}

For the proof of the following proposition we refer the reader to the proof of  \cite[Proposition 3.2]{DE}, with the following precautions.

\begin{enumerate}

\item All the rescales of the metric on the hypersurfaces  come from  a conformal changing on the metric of $\hyp^n\times\R$ by a constant conformal factor.
 
\item The convergences needed in the proof are guaranteed by Lemma \ref{compactness-lemma}.

\end{enumerate}

\begin{prop}\label{decay}
Let $M$ be a complete hypersurface in $\hyp^n\times \R$ with finite strong total curvature. Then, given $\varepsilon>0$ there
exists $R_0>0$ such that, for $R>R_0$,

\begin{equation}
R^2\sup_{x\in M-D_R(p)}|A|^2(x)<\varepsilon.
\end{equation}

where $D_R(p)$ is the intrinsic open $n$-ball of $M$  centered at a point $p\in M$ of radius $R.$

\end{prop}

\

 In Lemma \ref{identities} and Lemma \ref{proper-lemma} below, we explore the inclusion $\hyp^n\times \R\subset\R^{n+1}$ and we consider the canonical basis $\{e_1,\dots,e_n,e_{n+1}\}$ of $\R^{n+1}$ as a basis at each tangent plane of  $\hyp^n\times \R$.  Let $\Vert\; \Vert^2=\langle\;,\;\rangle$ and $\ov\nabla$ denote, respectively, the metric and the covariant derivative of $\hyp^n\times\R$. For a vector field $V(q)=\sum v_i(q)e_i+v_{n+1}(q)e_{n+1}$ in $\hyp^n\times\R$, where $q=\sum q_ie_i+q_{n+1}e_{n+1}\in \hyp^n\times \R$, we have

 \begin{equation*}
 \Vert V(q)\Vert^2=\frac{4\sum_{i=1}^nv_i^2(q)}{\left(1-\sum_{i=1}^nq_i^2\right)^2}+v^2_{n+1}(q),
  \ \ 
 \Vert q\Vert^2=\frac{4\sum_{i=1}^nq_i^2}{(1-\sum_{i=1}^nq_i^2)^2}+q^2_{n+1}.
 \end{equation*}
  
  Then, it is clear that: 
  
  \begin{equation}
  \label{limitada}
   \Vert q\Vert\to\infty\;\;  \mbox{iff}  \;\;q\to \partial_{\infty}(\hyp^n\times\R)
   \end{equation}

or, equivalently, that,  there exists $K_0>0$ such that  if $\Vert q\Vert<K_0,$ there exist $k_0>0,$ $t_0>0$  such that  ${\displaystyle\sum_{i=1}^n}q_i^2<k_0<1$  and 
$q_{n+1}^2<t_0.$

  In the next lemma, we consider the vector field  $X(p)=\sum x_i(p)e_i+x_{n+1}(p)e_{n+1}$, $p\in M$, given by the immersion  and establish  some elementary useful identities.
 
 \
 
%  It is easy to see that 
%  \begin{center}
%  $\Vert q\Vert\to\infty\;\;$ iff $\;\;q\to \partial_{\infty}(\hyp^n\times\R)$
%  \end{center}
%  or, equivalently, that
%  
% % \begin{center}
% \begin{equation}\mbox{{\small  there exists $K_0>0$ such that $\Vert q\Vert<K_0$ iff  there exist $k_0,t_0$ such that ${\displaystyle\sum_{i=1}^nq_i^2}<k_0<1$ and $\;q_{n+1}^2<t_0$.}}
% \label{limitada}
% \end{equation}
% % \end{center}
%  In the next lemma, we consider the vector field  $X(p)=\sum x_i(p)e_i+x_{n+1}(p)e_{n+1}$, $p\in M$, given by the immersion  and establish  some elementary useful identities.
% 
% \

\begin{lem}
 We have
 \begin{itemize}
 \item [(i)]$\ov\nabla_{e_j}X=Le_j$ and $ \ov\nabla_{e_{n+1}}X=e_{n+1}$, where  $L=(1+\displaystyle{\sum_{i=1}^{n}}x_i^2)(1-\displaystyle{\tiny{\sum_{i=1}^{n}}}x_i^2)^{-1}\geq 1$.
 \item [(ii)]$\ov\nabla_TX=\displaystyle{L\sum_{j=1}^n}t_je_j+t_{n+1}e_{n+1}$, where $T=\displaystyle{\sum_{i=1}^n} t_ie_i+t_{n+1}e_{n+1}$. 
 \end{itemize}
\label{identities}
\end{lem}

\begin{proof}

We first recall that the coefficients of the metric in $\hyp^n\times\R$ are given by $g_{ij}=\frac{\delta_{ij}}{F^2},$   $g_{n+1,n+1}=1,$ 
$g_{i,n+1}=0,$ where $i,j=1,\dots,n$ and $F=\frac{1}{2}(1-\sum_{i=1}^nx_i^2)$. Then a straightforward computation gives that the Christoffel symbols satisfy $\Gamma_{ij}^k=0$ if $i,j,k$ are all distinct or if at least one of the indices is $n+1.$  Moreover 

\begin{equation*}
\Gamma_{ij}^i=\frac{x_j}{F},\ \Gamma_{ii}^j=-\frac{x_j}{F}, \ \Gamma_{ij}^j=\frac{x_i}{F},\ \Gamma_{ii}^i=\frac{x_i}{F}, \mbox{ with } i,j=1,\dots,n. 
\end{equation*}

 Then, for $j\leq n$ we have $\ov\nabla_{e_j}X=e_j+\sum_{i=1}^n x_i\ov\nabla_{e_j}e_i$  and

 \begin{align} 
 \sum_{i=1}^n x_i\ov\nabla_{e_j}e_i&= \sum_{i=1}^n x_i\sum_{k=1}^{n+1}\Gamma_{ji}^ke_k\notag\\
 &=\sum_{i\not=j}^nx_i\sum_{k=1}^{n}\Gamma_{ji}^ke_k+x_j\sum_{k=1}^{n}\Gamma_{jj}^ke_k\notag\\
 &=\sum_{i\not=j}^nx_i[\Gamma_{ji}^je_j+\Gamma_{ji}^ie_i]+x_j[\Gamma_{jj}^je_j+\sum_{i\not=j}\Gamma_{jj}^ie_i]\notag\\
 &=\sum_{i\not=j}^n\frac{x_i^2}{F}e_j+ \sum_{i\not=j}^n\frac{x_ix_j}{F}e_i+\frac{x_j^2}{F}e_j+\sum_{i\not=j}^n-\frac{x_ix_j}{F}e_i\notag\\
 &=\sum_{i=1}^n\frac{x_i^2}{F}e_j.
 \end{align}

 Summing up, we obtain

 \begin{equation*}
 \ov\nabla_{e_j}X=e_j+\sum_{i=1}^n\frac{x_i^2}{F}e_j=\frac{1+\sum_{i=1}^{n}x_i^2}{1-\sum_{i=1}^{n}x_i^2}e_j=Le_j.
  \end{equation*}

The equality  $\ov\nabla_{e_{n+1}}X=e_{n+1}$   is straightforward and finishes the proof of {\em (i)}. 

The proof of  {\em(ii)}, is straightforward

\begin{equation*}
\ov\nabla_TX=\sum_{i=1}^nt_j\ov\nabla_{e_j}X+t_{n+1}\ov\nabla_{e_{n+1}}X=L\sum_{j=1}^n t_je_j+t_{n+1}e_{n+1}.
\end{equation*}

\end{proof}

\

 In Lemma \ref{proper-lemma},  we  generalise \cite[Lemma 3.1]{DE} .

\begin{lem}\label{proper-lemma}
Let $X: M^n\to\hyp^{n}\times\R$ be a complete  isometric immersion  with finite   strong total curvature.  Then $X$  is proper 
% and the gradient $\nabla Q(p)$ of $Q$
%in $M$ satisfies
%$$
%\Vert \nabla Q (q)\Vert  \geq c>0
%$$
%where $c$ is a positive constant, for any $q$ outside  
% a ball of  $\hyp^{n}\times\R$, centered at the origin, of radius $r_0.$
%In particular, 
 and the extrinsic distance has no critical points outside 
 a ball of  $\hyp^{n}\times\R$.

\end{lem}

\begin{proof}

%Notice that the assumption that $(\sigma,0)\in X(M)$ is not a  loss of generality.
If the immersion is not proper, we can find a ray $\gamma(s)$ issuing
from the origin $(\sigma, 0)$, parametrized by the arc length $s$,  such that, as $s$
goes to infinity, $Q(s)$ is bounded (see the discussion before Lemma \ref{identities}), where $X(s)=X(\gamma(s))$ and   
$Q(s)=\Vert X(s)\Vert$.  Setting $T(s)=\gamma'(s)$, we have  

\begin{equation}
Q(s)\geq \langle X(s), T(s)\rangle.
\end{equation}

In order to estimate $\langle X,T\rangle$  from below, we start by estimating  

\begin{equation}
T\langle X,T\rangle=\langle\ov{\nabla}_TX,T\rangle+\langle X,\ov{\nabla}_T T\rangle.
\label{conta18}
\end{equation}

Let us estimate the first term in the right hand side of (\ref{conta18}). By using Lemma \ref{identities} and that $\Vert T(s)\Vert=1$ we have

\begin{equation}
\langle \ov{\nabla}_TX, T\rangle=t_{n+1}^2(1-L)+L
\geq 1-L+L=1,
\label{ineq-proper0}
\end{equation}

where we used that $t_{n+1}^2\leq 1$ and that   $L\geq 1.$ 

\

Now we estimate the second term in the right hand side of (\ref{conta18}). We first notice that, since $\gamma$ is a geodesic in $M$, the tangent component of
$\ov{\nabla}_TT$ vanishes and we have

\begin{equation}
\ov{\nabla}_TT=\langle\ov{\nabla}_TT,N\rangle
N=-\langle\ov{\nabla}_TN,T\rangle N=\langle A(T),T\rangle N.
\end{equation}

It follows, by Cauchy-Schwarz inequality, that

\begin{equation}
\label{ineq-proper01}
|\langle X,\ov{\nabla}_TT\rangle|\le\Vert X\Vert\,\Vert A(T)\Vert\,\Vert T\Vert\le \Vert X\Vert\,|A|.
\end{equation}

In view of \eqref{limitada}, since we are assuming that $Q(s)$ is bounded, there exist $k_0$ such that ${\displaystyle\sum_{i=1}^nx_i^2}<k_0<1$, then we obtain 

\begin{equation}
\label{ineq-proper01.5}
 \Vert X\Vert^2< \left(\frac{2}{1-k_0}\right)^2\sum_{i=1}^nx_i^2+x^2_{n+1}< \left(\frac{2}{1-k_0}\right)^2\left( \sum_{i=1}^nx_i^2+x^2_{n+1}\right)=\left(\frac{2}{1-k_0}\right)^2|X|^2_{\tiny \R^{n+1}}.
\end{equation}

hence, by replacing \eqref{ineq-proper0},\eqref{ineq-proper01} and \eqref{ineq-proper01.5} in  \eqref{conta18}, we get

\begin{equation}
\label{ineq-proper1}
T\langle X,T\rangle\ge1-\frac{2}{1-k_0}| X|_{\tiny \R^{n+1}}\,|A|.
\end{equation}

 Now, we notice that $\Vert V\Vert > |V|_{\tiny \R^{n+1}}$, for all vector field $V$ tangent to $\hyp^n\times\R$.  In fact, $$
 \Vert V(q)\Vert^2=\frac{4\sum_{i=1}^nv_i^2(q)}{\left(1-\sum_{i=1}^nq_i^2\right)^2}+v^2_{n+1}(q)\geq \sum_{i=1}^nv_i^2(q)+v^2_{n+1}(q)=|V|^2_{\tiny \R^{n+1}}
  $$
 This implies, since $\gamma$ is a minimizing geodesic, that 
 $$s=\mbox{distance}_M ((X(s),(\sigma,0))>\mbox{distance}_{\;\R^{n+1}}(X(s),(\sigma,0))=| X|_{\tiny \R^{n+1}}$$
 
 which together with \eqref{ineq-proper1} gives
 
 $$
 T\langle X,T\rangle> 1-\frac{2}{1-k_0}s\,|A|.
 $$
 
  By using Proposition \ref{decay} with $\ve=\left(\frac{1-k_0}{2m}\right)^2$ we obtain

\begin{equation}\label{ineq-proper2}
T\langle X,T\rangle(s)> 1-\frac1m,
\end{equation}

for all $s>R_0$, where $R_0$ is given by Proposition
\ref{decay}. Integration of \eqref{ineq-proper2} from $R_0$ to $s$
gives
\begin{equation}\label{ineq-proper3}
\langle X,T\rangle(s)>\left(1-\frac1m\right)(s-R_0)+\langle
X,T\rangle(R_0).
\end{equation}

Since  $Q(s) = \Vert X(s)\Vert \geq \lgg X,T\rg(s)$, we see from  \eqref{ineq-proper3} that
$Q$ goes to infinity with $s$. This is a contradiction and proves
that $M$ is properly immersed.

\

Now we use that  $\hyp^n\times\R$ is a Hadamard manifold and we notice that Proposition \ref{decay} implies that $X: M^n\to\hyp^{n}\times\R$ has tamed second fundamental form (see \cite[Definition 1.1]{BC}). Then can use (the proof of) \cite[Theorem 1.2]{BC}  to conclude  there exists 
 a ball of  $\hyp^{n}\times\R$, centered at the origin, of radius $r_0$ such that the extrinsic  distance  has no critical points outside 
this ball.

\end{proof}

\begin{rem}
The technique of the proof of \cite[Theorem 1.2]{BC}  can also be used to get an alternative  proof of the fact  that $X$ is properly immersed.
\end{rem}

Let  $X:M\to\hyp^n\times\R$ be a hypersurface with  finite strong total curvature. By Lemma \ref{proper-lemma} there exists $r_0>0$ such that  
the distance function in $\hyp^n\times\R$ has no critical points in
$W=X(M)-(B_{r_0}(p_0)\cap X(M))$, where $B_{r_0}(p_0)$ is an extrinsic ${\rm (n+1)}$-ball of radius $r_0$ in $\hyp^n\times \R$. By Morse Theory, $X^{-1}(W)$  is
homeomorphic to $X^{-1}[X(M)\cap S_{r_0}(p_0)]\times[0,\infty)$, where $S_{r_0}(p_0)=\partial B_{r_0}(p_0)$.

\

 An {\em end}   $E$ of $M$ is a connected component of $X^{-1}(W)$. It follows that $M$ has only a finite number of
ends. In what follows, we identity $E$ and $X(E)$.

\

 With the same proof of \cite[Lemma (4.2)]{DE}, we can conclude that, for  $r>r_0$,  $E\cap B_{r_0}(p_0)$ is connected for each end $E$.
 
 Theorem \ref{strong-normal} below is a fundamental result for the characterization of finite strong total curvature hypersurfaces. We notice that it requires no assumption on $H_r$ and that it generalizes part of \cite[Theorem 1.1]{DE}.

 \begin{thm}
 \label{strong-normal}
 Let $X:M\to\h^n\times\R$, $n\geq 3$, be an orientable complete hypersurface    finite strong total curvature.  Then:

\begin{itemize}\itemsep=-1pt
\item[{\rm (i)}] The immersion $X$ is proper.

 \item[{\rm (ii)}] $M$ is diffeomorphic to a compact manifold $\ov M$ minus a finite
number of points $q_1, \dots q_k$.

\end{itemize}
 
 \end{thm}
 
 \begin{proof}

(i) has already been proved in Lemma \ref{proper-lemma}.  To prove (ii), we apply to each end $E_i$ the restriction of the  ambient transformation $I\colon ({\hyp}^n\times\R)-\{(\sigma,0)\}\to({\hyp}^n\times\R)-\{(\sigma,0)\}$, defined by $I(x)=x/\Vert x\Vert^2$, where the norm is with respect to the metric in ${\hyp}^n\times\R.$
 Then  $I(E_i)\subset B_1((\sigma,0))-\{(\sigma,0)\}$ and as $\Vert x\Vert \to\infty$ in $E_i$, $I(x)$ converges to the origin $(\sigma,0)$. It follows that each $E_i$ can be compactified with a point $q_i$. Doing this for each $E_i$, we obtain a compact manifold $\overline M$ such that $\ov M-\{q_1,\dots,q_k\}$ is diffeomorphic to $M$. This prove (ii).

 \end{proof}

\section{Finite strong total curvature and $H_r=0$}
\label{main-result}

 In the next theorem, we deal with an immersion $X:M\to\h^n\times\R$ with finite strong total curvature and $H_r=0$.   The proof  is inspired by the proof of  \cite[Theorem 2.1]{ST}, although the assumptions and the result are different in nature.

Let $\Pi_1,\dots\Pi_k$  be an admissible collection of hyperplanes of  $\hyp^n$, $P_i$, $i=1,\dots,k,$ the corresponding vertical hyperplanes and let $C^i_{\rho}$ the $\rho$-cylinder associated to  $\Pi_i,$ $i=1,\dots,k,$  as defined at the end of Section 1. 
We say that $M$  is {\em asymptotically close} to $(P_1\cup\dots\cup P_k)\times\R$ if for any $\rho,$ there is a compact  subset  $K_{\rho}$ of $M$ such that 

\begin{equation}
X(M\setminus K_{\rho})\subset \cup_{i=1}^kC^i_{\rho}.
\end{equation}

We   notice that, there are different notions of closeness at infinity and convergence in \cite{HNST, MMR, ST, ST1}. 

\

\begin{thm} 
\label{strong-normal-minimal}
Assume that $X:M\to\h^n\times\R$ has finite strong total curvature and satisfies $H_r=0$. Let $E$  be an end of $X(M)$ and let  $N=(N_1,\dots,N_{n+1})$ be  
a  unit normal vector field on $X(E).$ Let $\Pi_1,\dots\Pi_k$  be an admissible collection of hyperplanes of  $\hyp^n$ and $P_i$, $i=1,\dots,k,$ the corresponding vertical hyperplanes, such that $\partial E\subset \overline{P(\Pi_1,\dots,\Pi_k)}$. Suppose that
$\partial_{\infty}E\cap (\partial_{\infty}\hyp^n\times\R)\subset \partial_{\infty }(P_1\cup\dots\cup P_k)$.
% with $\partial_{\infty}E\cap (\partial_{\infty}\hyp^n\times\R)\cap \partial_{\infty } P_i\neq \emptyset$, for all $i.$ 
  Then:

\begin{itemize}
 \item[{\rm (i)}] $E$ is asymptotically close  to $P_1\cup\dots\cup P_k.$

 \item[{\rm (ii)}]   For any  sequence  of points $\{p_m\}\subset E$  converging  to a point in $\partial_{\infty}E,$ 
 the sequence $\{N_{n+1}(p_m)\}$ converges uniformly to zero.
\end{itemize}

 \end{thm}

\begin{proof}

We start by proving (i). 
 
 Let us first observe the following general facts:
 \begin{itemize}
 \item   By Corollary \ref{halfspace-coro} one has
\begin{equation}
E\subset \overline{P(\Pi_1,\dots,\Pi_k)}.
\end{equation}

\item 
Consider $\Pi_i$ and $\Pi_j,$ $i\not=j.$ Notice that two cases can happen.
\begin{itemize}
\item If $\partial_{\infty}\Pi_i\cap \partial_{\infty}\Pi_j\not=\emptyset,$ then for  any $\rho,$  $C^i_{\rho}\cap C^j_{\rho}\not=\emptyset.$
\item If  $\partial_{\infty}\Pi_i\cap \partial_{\infty}\Pi_j=\emptyset,$ then there exists $\rho$ such that   $C^i_{\rho}\cap C^j_{\rho}=\emptyset.$

\end{itemize}

\item Let  $\Theta$ be  a hyperplane in $\hyp^n$ disjoint from $\Pi_i$  and such that  $\partial E$ and $\cup_{i=1}^kP_i$ belong to the same component of $(\hyp^n\times\R)\setminus(\Theta\times\R).$ Denote by 
$(\Theta\times\R)^-$ the component of $(\hyp^n\times\R)\setminus(\Theta\times\R)$ that does not contain $\cup_{i=1}^k P_i\cup\partial E.$    Theorem  \ref{iperplano} yields that  
$E\cap (\Theta\times\R)^-=\emptyset.$ 

\item For each $i$, we can choose the corresponding equidistant hypersurface $L^{i +}_\rho$ to be the one which intersects $\overline{P(\Pi_1,\dots,\Pi_k)}$.

\end{itemize}

  Assume,  by contradiction, that there exists a positive number  $\rho$  such that $E^K:=E\setminus E\cap(\cup_{i=1}^k C^i_{\rho})$ is a non compact set. This means that there is an unbounded sequence of points $p_m=(x_m,t_m)\in E^K.$ Since $\{p_m\}$ is unbounded, we have  two possible cases. Either  there exists an $i\in\{1,\dots,k\},$ say $i=1,$ such that $\{x_m\}$  has a subsequence converging  to a point $\bar{x}$ of 
$\partial_{\infty}\Pi_1\cap(\partial_{\infty}\hyp^n\times \R)$  or $\{x_m\}$ is bounded  and $\{t_m\}$ is unbounded. 

Let us first  deal with the case (a subsequence of) $x_m$  converges to $\bar{x}$. Since $E^K\subset E\backslash C^1_{\rho} $ we can assume that the corresponding    $\{p_m\}$ is contained in $Z^{1+}_{\rho}$. We can choose a hyperplane $\Theta$ as above in $\hyp^n$  such that  $\Theta\cap  Z^{1+}_{\rho}\not=\emptyset$  and  
$\partial_{\infty}\Theta\cap\partial_{\infty}\Pi_1=\bar x.$ This leads to a  contradiction with  the fact that  $E\cap (\Theta\times\R)^-=\emptyset$  and then we must have that $\{x_m\}$ is bounded.

Now, let $p_m=(x_m,t_m)$ be a sequence in $E^K$ such that $x_m$ is bounded and $t_m$ is unbounded. Without loss of generality we may assume that  $t_m\longrightarrow\infty.$  
In this case, we get a  contradiction using the hypersurfaces ${\mathcal M}^r_d,$ $d>1,$ described in Theorem \ref{classification-theorem}, constructed with respect to one of the vertical hyperplanes $P_i, $ say $P_1$.
We can choose the family ${\mathcal M}^{r}_d$ such that each hypersurface contains the equidistant hypersurface $L^{1+}_{\rho},$ with $\rho=\cosh^{-1}(d),$  and that is   contained in the closure of $Z^{1+}_{\rho}.$
Let $V_{d}$ be the closure of the connected component of $(\hyp^n\times\R)\setminus {\mathcal M}_d^{r},$ not containing $P_1.$ By the properties of ${\mathcal M}^r_d,$ the height of $V_{d}$ is bounded.  Moreover, as  
$t_m\longrightarrow\infty,$ there exists $t>0$ such that $\partial E\cap V_d(t)=\emptyset$ and $E\cap V_d(t)\not=\emptyset,$ where $V_d(t)$ is the vertical translation of $V_d$ of height $t.$ 
Denote by ${\mathcal M}^{r}_d(t)$ the vertical translation of ${\mathcal M}_d^{r}$ of height $t.$
Let $\gamma$ be a geodesic $\hyp^n\times \{0\}$, orthogonal to $P_1$ at a point $p$, whose endpoint is a point $q\in\partial_\infty\hyp^n\times \{0\}$ that is outside all closed balls limited by $\partial\Pi_i$, $i=1,\ldots,n$. Such a point $q$ exists since $n\geq 3.$   Now, let us consider the horizontal translations along $\gamma$ (extended slice-wise to $\h^n\times\R$) of ${\mathcal M}^r_d(t)$, in the direction of $q$.  Since $E$ is properly immersed and $\partial_\infty E\cap (\{q\}\times\R)=\emptyset$, we can proceed as in the proof of Theorem \ref{iperplano} and we find a horizontal translation along $\gamma$  of  ${\mathcal M}^r_d(t)$  that has a last contact point with an interior point of $E$.  This is a contradiction by the maximum principle. 
Hence (1) is proved.

 Now we prove (2). 

 Assume, by contradiction, that there exist $\varepsilon>0$ and  a sequence of points $p_m=(x_m,t_m)$ converging to a point in $\partial_{\infty} E$  such that $|N_{n+1}(p_m)|>\varepsilon$. Since $\{p_m\}$ is unbounded, we have  two possible cases. Either  there exists $i=1,\dots,k,$ say $i=1,$ such that $\{x_m\}$  has a subsequence converging  to a point of 
$\partial_{\infty}\Pi_1\cap(\partial_{\infty}\hyp^n\times \R)$  or $\{x_m\}$ is bounded  and $\{t_m\}$ is unbounded.

Let $p_0\in E$ be a fixed point. Since $E$ has finite strong total curvature,  Proposition \ref{decay} implies that there exist  $R_0>0$ and $s>0$ such that  

\begin{equation}
\sup_{x\in (E\setminus(E\cap B_{R_0}(p_0)))}|A|^2(x)<s.
\label{constant-s}
\end{equation}

   where $B_{R_0}(p_0)$ is an extrinsic ${\rm (n+1)}$-ball of radius $R_0$ in $\hyp^n\times \R$.   Notice that,  in the previous inequality, we can take the extrinsic ball, because $E$ is properly immersed.

Assume  first that (a subsequence of) $x_m$  converges to $\bar{x}\in\partial_{\infty}\Pi_1\cap(\partial_{\infty}\hyp^n\times \R).$  For the following constructions, see Figure 2.

\begin{figure}[!h]
\centerline{
\includegraphics[scale=0.4]{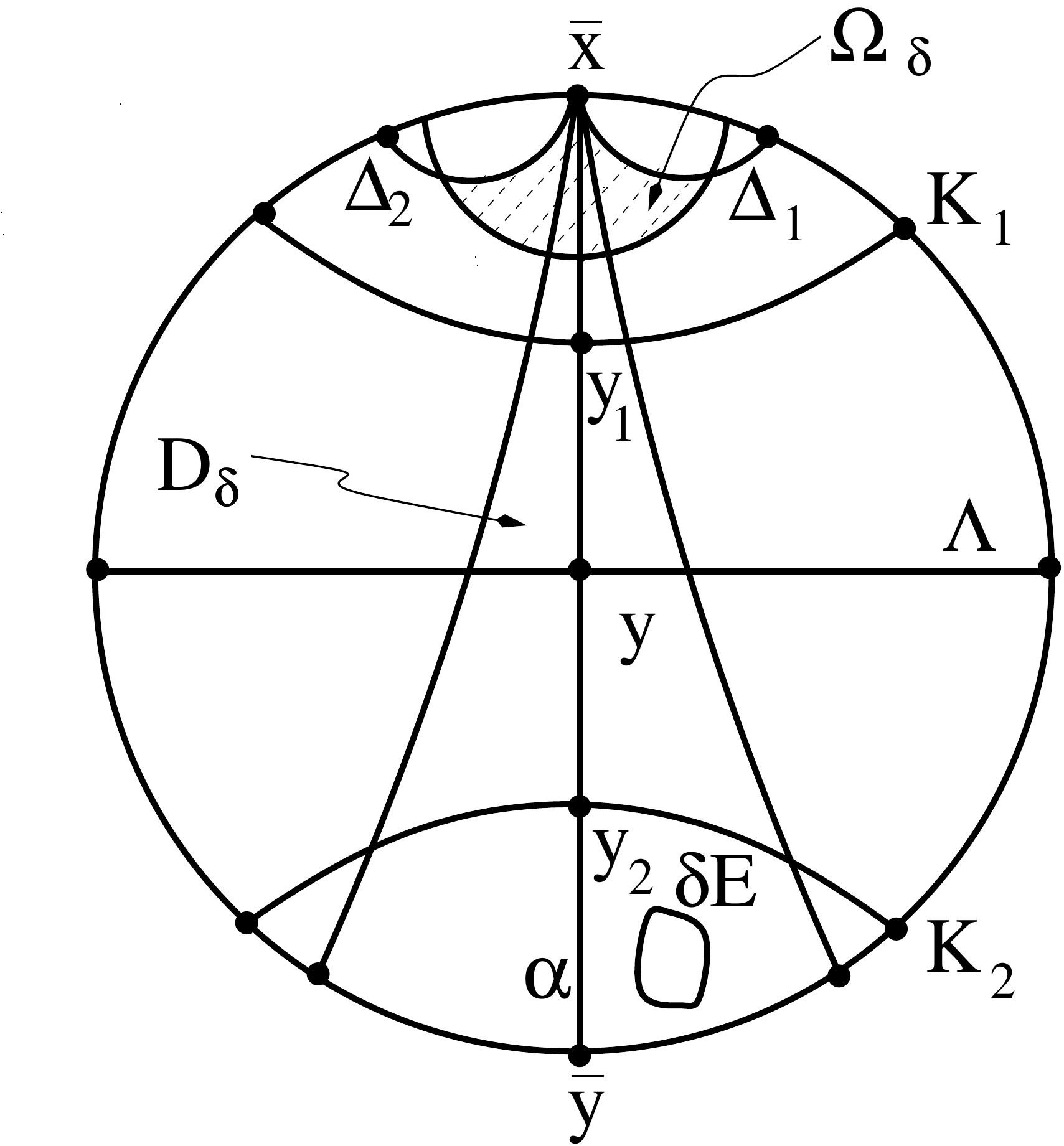}
\label{for-main-theo} 
}
\caption{Proof of Theorem 5.1}
\end{figure}

Without loss of generality, we can consider a geodesic $\alpha\subset \Pi_1$, passing through the origin,  such that 
$\bar x\in \partial_{\infty}\alpha$ and denote by $\bar y$ the point of $\partial_{\infty}\alpha $ distinct from $\bar x.$ 
We choose two points $y_1$ and $y_2$ on $\alpha$ such that $y_1$ is between $\bar x$ and $y_2.$ Let $y$ be the point on $\alpha$ equidistant from $y_1$ and $y_2.$ Finally, let  $\Lambda$  be the hyperplane in $\hyp^n$ through $y,$ orthogonal to $\alpha$ and  
 $K_i,$ $i=1,2$ be the  hyperplanes in 
$\hyp^n$ passing through $y_i,$ orthogonal to $\alpha.$ For $i=1,2$, denote by $(K_i\times\R)^+$ the connected component of 
$\hyp^n\times\R\setminus (K_i\times\R)$  whose asymptotic boundary contains $\bar y$ and by $(K_i\times\R)^-$ the other connected component. Since $\partial E$ is compact, it is possible to choose $y_1$ and $y_2$ such that $\partial E\subset (K_2\times\R)^+.$

Let ${\Delta}_1$ and ${\Delta}_2$  hyperplanes  in $\hyp^n$, symmetric  with respect to $\Pi_1, $ disjoints from $K_1$, such that, for $i=1,2$:
\begin{itemize}
\item $ \partial_{\infty }\Delta_i\cap\partial_{\infty }\Pi_1=\bar{x}.$
\item $\partial_{\infty}{\Delta}_i\cap\partial_{\infty}K_1=\emptyset$.
\item  $\partial E$ and $\cup_{i=1}^kP_i$ belong to the same component of $(\hyp^n\times\R)\setminus(({\Delta}_1\times\R)\cup({\Delta}_2\times\R)).$
\end{itemize}

 This yields that $({\Delta}_i\times\R)\subset (K_1\times\R)^-$, $i=1,2$, and by Corollary  \ref{halfspace-coro} we conclude that $E\cup\partial E$ is contained in the component of $\hyp^n\times\R\setminus (({\Delta}_1\times\R)\cup({\Delta}_2\times\R))$ containing $\Pi_1\times\R.$ Without loss of generality, we can choose ${\Delta}_i$ such that $({\Delta}_i\times\R)\cap E=\emptyset$.

\

For any $\lambda>0,$ we denote by $T_{\lambda}$ the hyperbolic translation of length $\lambda$ along $\alpha$ oriented from $\bar x$ to $\bar y.$ 
By abuse of notation, we also denote by $T_{\lambda}$ the extension of $T_{\lambda}$  to $\h^n\times\R$.  
For any $\lambda,$ denote by $U_{\lambda}$  the connected component of 
$\hyp^n\setminus (T_{\lambda}({\Delta}_1)\cup T_{\lambda}({\Delta}_2))$ containing  $\Pi_1.$ For any $\delta>0,$ there exists $\lambda(\delta)$ such that 
  the $(n-1)$-planes $T_{\lambda(\delta)}({\Delta}_i),$ $i=1,2,$  are contained in a neighborhood of $\alpha$ of diameter $\delta$  in the Euclidean metric in $\hyp^n$.
  
  \

 Let  $D_{\delta}$ be  the component of 
$\hyp^n\setminus(T_{\lambda(\delta)}({\Delta}_1)\cup T_{\lambda(\delta)}({\Delta}_2)\cup K_1\cup K_2)$ containing the point 
$y=\alpha\cap\Lambda.$ We notice that  $D_{\delta}\times \R=(U_{\delta}\times\R)\cap  (K_1\times\R)^+\cap (K_2\times\R)^-.$ Finally, denote by $\Omega_{\delta}$ 
the component of  $\hyp^n\setminus(T_{\lambda(\delta)}^{-1}(K_2)\cup {\Delta}_1\cup {\Delta}_2)$ such that $({\Omega}_{\delta}\times \R)\cap (\Pi_1\times\R)\not=\emptyset$   and 
$\partial_{\infty}\Omega_{\delta}=\bar x$. By construction, for any $\lambda>\lambda(\delta)$ and any $p\in\Omega_{\delta}\times\R,$ we have $T_{\lambda}(p)\in U_{\delta}\times\R.$ We notice that we can choose $\lambda(\delta)$ such that $ ({\Omega}_{\delta}\times \R) \subset (\hyp^n\times\R)\backslash B_{R_0}(p_0)$.

\

As $x_m\longrightarrow \bar x,$ we can assume that $p_m\in \Omega_{\delta}\times \R$ for   $m$ large. Moreover, for any $m$ large, there exists a unique 
$\lambda_m>0$ such that $T_{\lambda_m}(x_m)\in\Lambda,$ hence $q_m:=T_{\lambda_m}(p_m)\in \Lambda\times\R.$ 
For $m$ large enough, say $m>m_0,$ we have $\lambda_m>\lambda(\delta),$ which implies that $q_m\in(\Lambda\times\R)\cap (U_{\delta}\times\R)$.

For any $m>m_0,$ we denote by $E_m(\delta)$ the connected component of $T_{\lambda_m}(E)\cap (D_{\delta}\times\R)$ containing $q_m.$ By construction, 
$E_m(\delta)$ is the component of $T_{\lambda_m}(E\cap(\Omega_{\delta}\times\R))\cap (D_{\delta}\times\R)$ containing $q_m$ and for all $m>m_0$,  the boundary of 
$E_m(\delta)$ satisfies

 \begin{equation}
 \label{boundary-En-1}
 \begin{array}{c}
   \partial E_m(\delta)\subset \partial(D_\delta\times\R)\\
   \mbox{but}\\
  \partial E_m(\delta)\cap( (T_{\lambda(\delta)}({\Delta}_1)\times\R)\cup (T_{\lambda(\delta)}({\Delta}_2)\times\R))=\emptyset.
  \end{array}
 \end{equation}
 
% The last assertions implies that for $m$ large, say $m>m_1>m_0,$  
% 
% \begin{equation}
% d_m(p, \partial T_{\lambda_m}(E))>s_1 \ \ {\rm for \ any }\ p\in E_m(\delta)
% \end{equation}
%
%where $d_m$ is the intrinsic distance on $T_{\lambda_m}(E)$ and $s_1$ is a constant. 
Since $ ({\Omega}_{\delta}\times \R) \subset (\hyp^n\times\R)\backslash B_{R_0}(p_0)$, we can use (\ref{constant-s}) in order to conclude that  for all $\ p\in E_m(\delta)$, $m>m_0$ and $\delta>0$ it holds

\begin{equation}
\label{bounded-hyp-curv-1}
|A_m(p)|\leq s,
\end{equation}

where $A_m$ is the shape  operator  of  $E_m(\delta).$  As $D_{\delta}$ is compact, we can
look at $D_{\delta}\times\R$ as a subset of $\R^{n+1}$ where the metric   inherited from $\hyp^n\times\R$ and the Euclidean metric  are $C^1$ close. 
Then one can prove  that the norms of the second fundamental forms of $E_m(\delta)$ induced by the Euclidean and the hyperbolic metric are close (see Proposition 3.1 in the Appendix of \cite{ST}).

 As the norm of the second fundamental forms of $E_m(\delta)$ in the hyperbolic metric is  uniformly bounded (see  inequality \eqref{bounded-hyp-curv-1}), the same  holds for  the norm of second fundamental forms of the family $E_m(\delta)$ measured in the Euclidean metric.  By standard arguments, one can prove that this uniform bound implies the existence of a positive number $\eta$, independent on $m$ and $\delta$, such that a part $F_m$ of $E_m(\delta)$ is the Euclidean graph of a function $f_m$ defined on an  $n$-ball  of radius $\eta$ of the tangent hyperplane of $F_m$  at $q_m$. Moreover,  by applying vertical translations, we can assume the points $q_m$ are in a compact set of the Euclidean space and then  all the  functions $f_m$ have a uniform (Euclidean) $C^1$ bound (see, for instance,  the proof of  Lemma 2.2 in \cite{CM}).

Recall that we are assuming,  by contradiction, that $|N_{n+1}(p_m)|=|N_{n+1}(q_m)|>\varepsilon,$ for any $m.$ Then if we denote by 
$\nu$ the Euclidean unit normal vector, we have
 $\Vert \nu_{n+1}(q_m)\Vert >\varepsilon',$ for some positive $\varepsilon'$ (see the formula in the proof of Proposition 3.2 in \cite{ST3}). The last inequality   implies that the slope of the tangent planes of $E_m(\delta)$ at points $q_m$ is uniformly bounded from below. 
As the gradient  of the functions $f_m$  are uniformly bounded  and $\eta$ does  not depend on $\delta$ we can choose $\delta$ small enough such that the graph $F_m$ intersect $(T_{\lambda(\delta)}({\Delta}_1)\times\R)\cup (T_{\lambda(\delta)}({\Delta}_1)\times\R),$ that is in  contradiction with \eqref{boundary-En-1}. 
This finishes the proof in the case  where $\{x_m\}$  converges to a point of $\partial_{\infty}\hyp^n\times\R$.

\

In  the case where $t_m$ is unbounded   and $x_m$ is bounded the proof is somewhat easier. Without loss of generality, we can assume that $t_m\longrightarrow\infty.$ We proved before that for any $\rho$ there exists $t_{\rho}>0$ such that $E\cap \{|t|>t_{\rho}\}\subset \cup_{i=1}^kC^i_{\rho}.$  Then, there exists $i\in\{1,\ldots,n\}$, say $i=1$,  and a subsequence $t_{m_{1}}$, such that $t_{m_{1}}\in C^1_{\rho}$. Since we are assuming that $\{x_m\}$ is bounded, we may assume that $\{p_{m_{1}}\}\subset \omega\times\R$, where $\omega\subset \hyp^n\times\{0\}$ is a compact set.
Then, we proceed as in the former case, replacing %$T_{\lambda}({\Delta}_1)\times\R$ and $T_{\lambda}({\Delta}_2)\times\R$ by the hypersurfaces 
% $L_{\rho}^+\times\R,$ $L_{\rho}^-\times\R$ and 
$D_\delta$ by $\omega\cap C^1_{\rho}.$

\end{proof}

%{\color{teal}   I believe we can prove a stronger version of the Theorem, namely:  

%{\bf Theorem}\\
%Let $E$ be a properly immersed annular end in $\h^n\times\R$  with $H_r=0$ and  finite strong total curvature.Then the following hold
%\begin{enumerate}
%\item[]{{\rm (1)}}  Let $N=(N_1,\dots,N_{n+1})$ be  
%a normal unit vector field on $E.$  For any sequence of points $\{p_m\}\subset E$  converging    to a point in $\partial_{\infty}E\cap (\partial_{\infty}\hyp^n\times\R)$, the sequence $\{N_{n+1}(p_m)\}$ converges uniformly to zero. 
%\item []{{\rm (2)}} Assume that $\partial_{\infty}E\cap (\partial_{\infty}\hyp^n\times\R)\subset \partial_{\infty }P$, where 
%$P$  is a vertical  hyperplane in $\hyp^n\times\R.$ Then 
%$E$ is asymptotically close  to $P$. Moreover, if $\{p_m\}\subset E$ is a sequence of points  converging    to $\partial_{\infty}E$ then  the sequence $\{N_{n+1}(p_m)\}$ converges uniformly to zero.
%\end{enumerate}

%}

%\begin{thm}
%\label{strong-plane}
%Let $E$ be an  annular end with $H_r=0$  and with finite strong total curvature such that $\partial_{\infty}E\cap (\partial_{\infty}\hyp^n)\times\R\subset L_1\cup L_2,$ where 
%$L_1,\ L_2$ are two vertical $(n-1)$-planes contained in $\partial_\infty\hyp^n\times \R.$ 
%Then $E$ converges  to the $n$-plane  $P$ in $\hyp^n\times\R$ such that $\partial_{\infty}P\cap(\partial_{\infty}(\hyp^n)\times \R)=L_1\cup L_2.$ 

%\end{thm}

\begin{rem} 
\label{question}
Theorem \ref{strong-normal-minimal} can be viewed as a step towards a generalization of    the results of  \cite[Theorem 3.1 (c)]{HaRo} and \cite{HNST}  for minimal surfaces with finite total curvature in $\hyp^2\times\R.$ We point out that our technique is completely different  from the one in \cite{HaRo, HNST} where complex analysis is a key tool.

\end{rem}

% Let us finish by asking a question.   With  the further assumption of  stability, R. Sa Earp and E. Toubiana are able to prove that a minimal surface in $\hyp^2\times\R$ has   infinite total curvature, provided the asymptotic  boundary contains a curve that is not a vertical segment \cite[Theorem 4.1]{ST1}. The authors use stability to get {\em a-priori} estimates for the norm of the second fundamental form. 
%
%In our case,  the stability assumption does not make sense,  so we need to assume finite strong total curvature in order to get {\em a-priori}  estimates for the for the norm of the second fundamental form. The  following question arises:    could the hypothesis $\partial_{\infty}E\cap (\partial_{\infty}\hyp^n)\times\R\subset \partial_{\infty }(\Pi\times\R)$  be removed by  
%proving that if this is not true, then the strong total curvature is   not finite?
% In this case we could  relax the hypothesis in  Theorem \ref{strong-normal}.
%
%

\

{\em Acknowledgments.} The authors would like to thank the referee for the careful reading and the valuable suggestions.


\begin{thebibliography}{999999}

%\bibitem[AF]{AF} \textsc{R.A. Adams, J. J. F. Fournier:} {\em Sobolev Spaces}, Second Edition,  Elsevier (2003).

\bibitem[A]{An}\textsc{M. Anderson:} {\em The compactification of a minimal submanifold in Euclidean space by the Gauss map,} Preprint IHES (1985).

\bibitem[B]{B} \textsc{R. Bartnik:} {\em The mass of an asymptotically flat manifold},  Comm. Pure Appl.Math., 39, 661-693 (1986).

\bibitem[BC]{BC} \textsc{G.P. Bessa, M.S. Costa:} {\em On submanifolds with tamed second fundamental form,} Glasgow Mathematical Journal 51 (3) (2009), 669-680

%\bibitem[A-I-R]{A-I-R} Alias, L.J., Impera, D. and Rigoli, M.: Hypersurfaces of constant higher order mean curvature in warped products. \textit{Trans. Amer. Math. Soc.} \textbf{365}, (2013), 591-621.
%\bibitem [B-C]{B-C} \textsc{Barbosa, J.L. and Colares, A.G.:} Stability of hypersurfaces with constant r-mean
%curvature, \textit{Ann. Global Anal. Geom.} \textbf{15} (1997), 277-297.
%\bibitem [BFM]{BFM} \textsc{ J.L. Barbosa, R. Fukuoka, F. Mercuri:} {\em Immersions of finite geometric type in Euclidean spaces,}  
%Ann. Glob. Anal. Geom. 22 (4)  301-315 (2002).

\bibitem [BS1]{BS1}\textsc{P. Bérard,  R.  Sa Earp: } {\em Minimal hypersurfaces in $\hyp^n\times \R$, total curvature
and index,} Boll. Unione Mat. Italiana, 9 (3)  (2016) 341-362.
%\bibitem [BS2]{B-SE2} \textsc{ P. Bérard,  R.  Sa Earp:} {\em  Examples of H-hypersurfaces in $\hyp^n\times \R$ and geometric applications,} Matemática Contemporânea 34, (2008), 19-51.

%\bibitem[Ca] {Ca}\textsc{A. Carlotto:} {\em   Rigidity of stable minimal hypersurfaces in asymptotically flat spaces,}  
Calc. Var. Partial Differential Equations 55, 3 (2016).


\bibitem[CM]{CM}\textsc{T. H.  Colding, W. P. Minicozzi:} {\em Minimal surfaces,}  Courant Lecture Notes in Math. 4, New York University, 
Courant Institute of Mathematical Sciences, New York, (1999).
%\bibitem[CM1]{CM1}\textsc{T. H.  Colding, W. P. Minicozzi:} {\em TheCalabi-Yau conjectures for embedded surfaces,}  Ann. of Math., (167) 211- 243  (2008). 

%\bibitem[Col]{Col}\textsc{P. Collin:} {\em Topologie et courbure de surfaces minimales dans $\R^3,$} Ann. of Math 2 (145) 1-31 (1997). 


%\bibitem [C-P]{C-P} Colares, A.G. and Palmas, O.: Addendum to “Complete rotation hypersurfaces with $H_k$ constant in space forms. \textit {Bull.
%Braz. Math. Soc., New Series}, \textbf{39(1)} (2008), 11„1¤720.
%\bibitem [E1]{E1} Elbert, M.F.: On complete graphs with negative r-mean curvature. \textit{Proc. of the Amer. Math. Soc.} \textbf{128}, (5) (2000), 1443-1450.
\bibitem[DE]{DE}\textsc{M. Do Carmo, M.F. Elbert:} {\em Complete hypersurfaces in Euclidean spaces with finite strong total curvature,} To appear in Comm. Anal. Geom..

%\bibitem[E2]{E2}  Elbert, M.F.: Constant positive 2-mean curvature hypersurfaces. \textit{Illinois J. of Math.} \textbf{46}, (1) (2002), 247-267.
%\bibitem [E-SE]{E-SE} Elbert, M.F. and Sa Earp, R.:  All solutions of the CMC-equation in $\hyp^n\times \R$ invariant by parabolic screw motion. \textit{Annali di Mat. Pura e App. DOI 10.1007/s10231-012-0268-8}.
%\bibitem [E-N-SE]{E-N-SE} Elbert, M.F. Nelli, B. and Sa Earp, R.: Existence of vertical ends of mean curvature 1/2 in $\hyp^n\times\R$. \textit{Trans. Amer. Math. Soc..}, \textbf{364}(3), (2012), 1179-1191.
%\bibitem[E-SE-S]{E-SE-S} Elbert, M.F., Sa Earp, R. and Santos, W.: 1-minimal rotational hypersurfaces in $\hyp^n\times\R$. Preprint.
%\bibitem [E-G-R]{E-G-R} Espinar, J.M.,  G\'{a}lvez, J.A. and Rosenberg, H.: Complete surfaces with positive
%extrinsic curvature in product spaces. \textit{ Comment. Math. Helv.} \textbf{84}(2) (2009), 351„1¤7386.
%\bibitem[GM-I-R]{GM-I-R} García-Martínez, S.C., Impera, D. and Rigoli, M.: A sharp height estimate for compact hypersurfaces with constant k-mean curvature in warped product spaces. \textit{arXiv:1205.5628v2 [math.DG]}.
%\bibitem[H-L]{H-L} Hounie, J. and Leite, M.L.: Uniqueness and nonexistence theorems for hypersurfaces with $H_r= 0$. \textit{ Ann. of Global Anal. Geom.} \textbf{17} (1999), 397„1¤7407.
%\bibitem [F-S]{F-S} Fontenele, F.  and Silva, S.: A tangency principle and applications. \textit{Illinois J. of Math.} \textbf{54} (2001), 213„1¤7228.
%\bibitem [L1]{L1} Leite, M.L.: Rotational hypersurfaces of space forms with constant scalar curvature. \textit{Manuscripta Math.} \textbf{67} (1990), 285-304.
%\bibitem [L2]{L2} Leite, M.L.: The tangency principle for hypersurfaces with null intermediate curvature. \textit{XI Escola de Geometria Diferencial-UFF} Brazil, (2000).
\bibitem [ENS]{ENS} \textsc{M.F. Elbert, B. Nelli, W.  Santos:} {\em Hypersurfaces with $H_{r+1} = 0$ in $\hyp^n\times\R$,} Manuscripta Mathematica, 149 (2015) 507-521. 

%\bibitem[FMMR]{FMMR}\textsc{L. Ferrer, R. Mazzeo, F. Martin, M. Rodriguez:} {\em Properly embedded minimal annuli in 
%$\hyp^2\times\R,$} Preprint arXiv:1704.07788.
%

\bibitem[FS]{FS}\textsc {F.X. Fontenele, S. L. Silva:} {\em Maximum principles for hypersurfaces with vanishing curvature
functions in an arbitrary Riemannian manifold,}  Anais da Ac. Bras. de Ci. 74 (2) (2002) 199-205.

\bibitem[HR]{HaRo}\textsc{L. Hauswirth, H. Rosenberg:} {\em Minimal surfaces of finite total curvature in $\hyp\times\R,$} 
Workshop on Differential Geometry,  Mat. Contemp. 31 (2006), 65-80.

%\bibitem[HoMe]{HoMe}\textsc{D. Hoffman, W. Meeks III:} {\em The strong halfspace theorem
%for minimal surfaces,} Invent. Math., 101 (1990) 373-377.

\bibitem [HL1]{HL1} \textsc{ J. Hounie, M. L. Leite:} {\em  Two ended hypersurfaces with zero scalar curvature,}
 Indiana Univ. Math. Jour.  48 (1999), 817-882.


\bibitem [HL2]{HL2}\textsc{J. Hounie, M. L. Leite:} {\em The maximum principle for hypersurfaces with vanishing curvature functions,} J. Differential Geom.    41, 2 (1995), 247-258.
\bibitem[HNST]{HNST}\textsc{L. Hauswirth, B. Nelli, R. Sa Earp, E. Toubiana:} {\em A Schoen theorem for minimal surfaces in $\hyp^2\times\R,$}  Adv. Math.  274  (2015), 199-240.


\bibitem[H]{Hu}\textsc{ A. Huber:} {\em  On subharmonic functions and differential geometry in the large,}  Comment. Math. Helvetici, 32, 
181-206 (1957).

\bibitem[MMR]{MMR}\textsc{R. Mazzeo, F. Martin, M. Rodriguez:} {\em Minimal surfaces with positive genus and finite total curvature in $\hyp^2\times\R$,} Geometry and Topology 18 (2014) 141-177.



\bibitem[MP]{MePe}\textsc{ W. H. Meeks III, J. Perez:} {\em A survey on classical minimal surface theory,} University Lecture Series (AMS) vol. 60 (2012).

%\bibitem[MeRo]{MeRo}\textsc{ W. H. Meeks III, H. Rosenberg:} {\em The uniqueness of the helicoid,} Ann. of Math. 161 (2005) 723-754.
%
%\bibitem[MR]{MR}\textsc{F. Morabito, M. Rodriguez:} {\em Saddle Towers and minimal $k$-noids in $\hyp^2\times\R,$}
%Jour. of the Inst. of Math.  Jussieu, 11 (2) (2012) 333-349. 

\bibitem [NST]{N-SE-T} \textsc{B. Nelli, R. Sa Earp, E. Toubiana:} {\em  Maximum principle and symmetry for minimal hypersurfaces in $\hyp^n\times \R$,} Annali della Scuola Normale Superiore di Pisa, Classe di Scienze (5) Vol. XIV (2015) 387-400.

%\bibitem[Oss1]{Oss1}\textsc{R. Osserman:}{\em Global properties of minimal surfaces in ${\mathbb E}^3$ and ${\mathbb E}^n,$} Ann. of Math. 
%2 (80) (1964) 340-364.

\bibitem[O]{Oss2}\textsc{R. Osserman:}{\em A Survey of Minimal Surfaces,} Dover Publications, New York, 2nd edition, (1986).


%\bibitem [P]{P}  O. Palmas.: Complete rotation hypersurfaces with Hk constant in space forms. \textit {Bull.
%Braz. Math. Soc.}, \textbf{30(2)} (1999), 139-161.

%\bibitem[P]{P}\textsc{J. Pyo:} {\em  New Complete Embedded Minimal Surfaces in $\hyp^2\times\R,$}  Ann. Global Anal. Geom. 40, 2 (2011) 167-176.




\bibitem[ST]{ST}\textsc{R. Sa Earp, E. Toubiana:} {\em A minimal stable vertical planar end in $\hyp^2\times\R$ has finite total curvature,} 
Jour. of the London Math. Soc. 92, 3  (2015) 712-723. 
\bibitem [ST1]{ST1}  \textsc{R. Sa Earp, E. Toubiana: }  {\em Concentration of total curvature of minimal surfaces in $\hyp^2\times\R,$}   Math. Annalen, 369, Issue 3-4  (2107) 1599-1621.
\bibitem [ST2]{ST2}  \textsc{ R. Sa Earp, E. Toubiana: } {\em  An asymptotic theorem  for minimal surfaces and existence results for minimal graphs in  $\hyp^2\times \R$,} Math. Ann.  342 (2008)
309-331.
\bibitem[ST3]{ST3}\textsc{R. Sa Earp, E. Toubiana:} {\em Minimal graphs in $\hyp^n\times\R$ and $\R^{n+1},$} Ann. Inst.  Fourier 60 (7) (2010) 2373-2402.

%\bibitem [SY]{SY} Schoen, R. and Yau, S.T.: Lectures on Differential Geometry. International Press, (1994).




\bibitem[Wh]{Wh}\textsc{B. White:} {\em Complete surfaces of  finite total curvature,} J. Diff. Geometry, 26, 315-326 (1987).

\end{thebibliography}
\end{document}